\documentclass[a4paper,12pt]{article}
\usepackage[top=2.5cm,bottom=2.5cm,left=2.5cm,right=2.5cm]{geometry}
\usepackage{cite, amsmath, amssymb}
\usepackage{amssymb}
\usepackage{graphicx}
\newenvironment{proof}{{\noindent \it Proof.}}{\hfill $\blacksquare$\par}
\usepackage{latexsym}
\usepackage{CJK}
\usepackage{graphicx,booktabs,multirow}
\usepackage{tikz}
\usetikzlibrary{decorations.pathreplacing}
\usetikzlibrary{intersections}
\usepackage{enumerate}
\usepackage{graphicx,booktabs,multirow}
\usepackage{appendix}
\newtheorem{theorem}{Theorem}[section]
\newtheorem{proposition}[theorem]{\rm\bfseries Proposition}
\newtheorem{lemma}[theorem]{Lemma}

\newtheorem{remark}[theorem]{Remark}

\newtheorem{problem}{Problem}[section]

\usepackage[section]{placeins}
\usepackage{setspace}
\numberwithin{equation}{section}
\allowdisplaybreaks
\newcommand \dbar {{d\mkern-8mu\mathchar'26\mkern-1mu}}
\def\NAT@def@citea{\def\@citea{\NAT@separator}}
\usepackage{indentfirst}
\begin{document}
\vspace*{10mm}

\noindent
{\Large \bf Signed graphs with exactly two main eigenvalues: The unicyclic case}

\vspace*{7mm}

\noindent
{\large \bf Zenan Du$^{1,*}$, Fenjin Liu$^2$, Hechao Liu$^3$, Jifu Lin$^4$, Wenxu Yang$^5$}
\noindent

\vspace{7mm}

\noindent
$^1$ School of Mathematics and Statistics, Shanxi University, Shanxi, 030006, P. R. China,\\
$^2$ School of Science, Changan University, Xian, 710064, P.R. China,\\
$^3$ School of Mathematics and Statistics, Hubei Normal University, Huangshi, 435002, P. R. China, \\
$^4$ School of Mathematical Sciences, South China Normal University,  Guangzhou, 510631, P. R. China,\\
$^5$ Taiyuan Branch, Shennong Technology Group Co., Ltd., Shanxi, 030006, P. R. China\\
E-mail: {\tt duzn@sxu.edu.cn}, {\tt fenjinliu@163.com}, {\tt hechaoliu@yeah.net},\\ \quad{\tt 2023021893@m.scnu.edu.cn}, {\tt amber130120@icloud.com} \\[2mm]
$^*$ Corresponding author
\noindent

\vspace{7mm}

\noindent
{\bf Abstract} \
\noindent
An eigenvalue $\lambda$ of a signed graph $S$ of order $n$ is called a main eigenvalue if its eigenspace is not orthogonal to the all-ones vector $j$. Characterizing signed graphs with exactly $k$ $(1\le k\le n)$ distinct main eigenvalues is a problem in algebraic and graph theory that has been studied since 2020. Du et al. (2024, 2026) characterized a class of signed graphs with exactly two main eigenvalues by analyzing a type of multigraph whose base graph is a tree. In this paper, we extend this study to the case where the associated multigraph has a unicyclic base graph, and we conclude by proposing several open problems.
 \\[2mm]


\baselineskip=0.30in

\section{Introduction}\label{sec-introduction}
\begin{spacing}{1.3}
Let $G$ be a simple graph (without multiple edges and self-loops) of order $n$ with vertex set $V(G)$ and edge set $E(G)$. If the vertices $v_i$ and $v_j$ are adjacent, we write $v_i\sim v_j$ (otherwise, $v_i\nsim v_j$), then $e=v_iv_j$ is an edge belonging to $E(G)$. For $\sigma: E(G)\rightarrow\{-1, +1\}$, $S=(G, \sigma)$ is called a \emph{signed} graph derived from its underlying graph $G$ with every edge having a positive or negative sign.

A graph $H$ is called a \emph{multigraph} if $E(H)$ is not necessarily composed of distinct pairs of vertices, then multiple edges are allowed and thus $E(H)$ is a multi-set. For a multigraph $H$, let $w_H(u,v)$ ($w(u,v)$ for short) be the number of edges between the vertices $u$ and $v$ in $H$. Then a simple graph $G$ or a signed graph $S$ can be regarded as a special multigraph with $w(u,v)=1$ for all edges $uv$. For any vertex $v\in V(H)$, let $N_H(v)$ ($N(v)$ for short) be the neighborhood set (composed of distinct vertices) of $v$ in $H$ and $d_H(v)$ ($d(v)$ for short) be the number of edges incident with $v$ in $H$. Clearly, $|N(v)|\leq d(v)=\sum\limits_{u\sim v} w(u,v)$, and equality holds if and only if $w(u,v)=1$ holds for any $u\sim v$. If $w(u,v)\in\{0,1,2\}$ holds for any $u,v\in V(H)$, then we say $H$ is a \emph{$(0,1,2)$-multigraph}. Let $H^\circ$ be the \emph{basic graph} (\emph{B-graph} for short) of a multigraph $H$, where $H^\circ$ is obtained from $H$ by replacing parallel edges with a single edge. Then $H^\circ$ is a simple graph with $V(H)=V(H^\circ)$, and thus $H$, $H^\circ$ have the same connectivity.

Let $S$ be a signed graph with its underlying graph $G$ of order $n$. The \emph{adjacency matrix} $A(S)=(s_{ij})$ of $S$ is a symmetric $(0,1,-1)$-matrix, where $s_{ij}=1$ if $v_i\sim v_j$ and $v_iv_j$ is a positive edge, $s_{ij}=-1$ if $v_i\sim v_j$ and $v_iv_j$ is a negative edge, $s_{ij}=0$ if $v_i\nsim v_j$. The \emph{eigenvalues} of $S$ are the eigenvalues of its adjacency matrix $A(S)$.

Let $M$ be a real symmetric matrix. An eigenvalue $\lambda$ of $M$ is said to be a \emph{main eigenvalue} if its \emph{eigenspace} $\varepsilon_M(\lambda)$ is not orthogonal to the all-ones vector $j=[1,1,\dots,1]^T$, and this definition does not depend on the multiplicity of $\lambda$. The main eigenvalues of a simple graph $G$, a signed graph $S$, and a multigraph $H$ are the main eigenvalues of its adjacency matrix $A(G)$, $A(S)$, and $A(H)$, respectively.

In 1978, Cvetkovi\'{c} has noticed that $G$ has exactly one main eigenvalue if and only if $G$ is regular. Besides, he posed the following long-standing problem: characterizing the graphs with exactly $k$ $(2\leq k \leq n)$ distinct main eigenvalues \cite{Cvetkovic-Main}. So far, this problem has attracted the attention of many scholars. For graphs of order $n$ with exactly two distinct main eigenvalues, one can see \cite{Feng,Hagos,Hou-bicyclic,Hou-trees,Hou-unicyclic,Hayat,Lepovic}; for graphs of order $n$ with exactly $n-1$ distinct main eigenvalues, one can see \cite{Du-1,Du-2,Du-3,Li,Qiu,Wang}; and for graphs of order $n$ with $n$ distinct main eigenvalues, one can see \cite{Cve-Controllable,Cve-Controllable-2,Stanic-2,Farrugia}, etc.

An \emph{automorphism} of a graph $G$ is a permutation $\pi$ of the vertex set $V(G)$ such that the pair of vertices $v_i\sim v_j$ if and only if $\pi(v_i)\sim\pi(v_j)$. The set of automorphisms of $G$ under the composition operation, form a group, called the \emph{automorphism group} of $G$ and denoted by ${\rm Aut}(G)$. Two vertices $v_i,v_j\in V(G)$ belong to the same \emph{orbit} if there is an automorphism $\pi\in {\rm Aut}(G)$ such that $\pi(v_i)=v_j$. In 1999, Cvetkovi\'{c} and Fowler \cite{Cve-group} showed that the number of distinct main eigenvalues of a graph $G$ does not exceed the number of orbits into which $V(G)$ is partitioned by ${\rm Aut}(G)$.

For a simple graph $G$ of order $n$, a \emph{walk} of length $k$ in $G$ from vertex $u$ to vertex $v$ is a sequence of vertices $v_1v_2\dots,v_kv_{k+1}$ (not necessarily distinct) such that $u=v_1,v=v_{k+1}$ and $v_i\sim v_{i+1}$ for each $1\le i\le k$. The \emph{walk-matrix} of $G$ is the $n\times n$ matrix
$W(G)=\begin{bmatrix}
    j,Aj,A^2j,\cdots,A^{n-1}j
  \end{bmatrix}$, where $A=A(G)$.
It is worth mentioning that, in 2002, Hagos \cite{Hagos} showed that the number of distinct main eigenvalues of a graph $G$ is equal to the rank of its walk-matrix $W(G)$.

Based on the research of main eigenvalues of simple graphs, a natural problem is characterizing the signed graphs with exactly $k$ ($\leq n$) distinct main eigenvalues. Let $d_{v_i}^+$ $(d_{v_i}^-)$ be the number of positive (negative) edges incident with $v_i$ in a signed graph $S$. The difference $d_{v_i}^+-d_{v_i}^-$ is called the \emph{net-degree} of $v_i$ and $S$ is called \emph{net-regular} if all vertices of $S$ have the same net-degree. In 2020, Stani\'{c} \cite{Stanic} has noticed that a signed graph has exactly one main eigenvalue if and only if it is net-regular, and in recent years, the papers \cite{Akbari,Du-4,Shao,Andelic} studied the main eigenvalues of signed graphs.

Let $S$ be a signed graph, $\tilde{S}$ be the \emph{associated multigraph} of $S$ with adjacency matrix $A(\tilde{S})=J-I-A(S)$, that is, for any vertices $u,v\in V(\tilde{S})$, $w_{\tilde{S}}(u,v)=i$ if $(A(\tilde{S}))_{uv}=i$ for $i\in\{0,1,2\}$. Since $A(S)$ is a symmetric $(0,1,-1)$-matrix, $A(\tilde{S})$ is a symmetric $(0,1,2)$-matrix, so we call the corresponding associated multigraph $\tilde{S}$ a $(0,1,2)$-multigraph. According to definitions of $S$ and $\tilde{S}$, Du et al. \cite{Du-4} obtained the following results.

\begin{proposition}\label{prop-one-to-one}{\rm (\cite{Du-4})}
  Let $S$ be a signed graph of order $n$, $\tilde{S}$ be the associated $(0,1,2)$-multigraph of $S$. Then \\
  {\rm (1)} there exists a bijection between the set of signed graphs of order $n$ and the set of $(0,1,2)$-multigraphs of order $n$;\\
  {\rm (2)} $S$ and $\tilde{S}$ have the same number of distinct main eigenvalues.
\end{proposition}
\vskip0.2cm

By Proposition $\ref{prop-one-to-one}$, one can study the signed graphs of order $n$ with exactly $k$ $(\leq n)$ distinct main eigenvalues by studying $(0,1,2)$-multigraphs of order $n$ with exactly $k$ $(\leq n)$ distinct main eigenvalues. In \cite{Du-4}, the authors studied the sufficient and necessary conditions for $k=2$ based on Theorem $2.1$ in \cite{Stanic}.

\begin{theorem}\label{lemma-s(v)}{\rm (\cite{Du-4})}
  Let $H$ be a multigraph of order $n$ $(\geq2)$ and $A$ $(=A(H))$ be its adjacency matrix, $s(v)$ be the number of walks of length $2$ in $H$ that starting at $v\in V(H)$. Then the following three conditions are equivalent.\\
  {\rm (1)} $H$ has exactly two distinct main eigenvalues.\\
  {\rm (2)} $j$ is not an eigenvector of $A$ and $(A^2-aA-bI)j=0$ holds for some integers $a,b$.\\
  {\rm (3)} There exist integers $a,b$ such that \begin{equation}\label{eq-s(v)-2}
  ad(v)+b=s(v)=\sum\limits_{u\sim v} w(u,v)d(u)
  \end{equation}
  holds for any vertex $v\in V(H)$ and $j$ is not an eigenvector of $A$.

  In particular, if $H$ is a $(0,1,2)$-multigraph of order $n$ $(\geq3)$ and the $B$-graph of $H$ is a tree, then $H$ has exactly two distinct main eigenvalues if and only if there exist integers $a,b$ such that \eqref{eq-s(v)-2} holds for any vertex $v\in V(H)$.
\end{theorem}


\begin{proposition}\label{prop-a-b-fenlei}{\rm (\cite{Du-4})}
  Let $H$ be a $(0,1,2)$-multigraph of order $n$ $(\geq3)$ with exactly two distinct main eigenvalues, the $B$-graph of $H$ be a tree, $P=v_0v_1\cdots v_l$ be the longest path of $H$, and $a,b$ be integers that satisfy \eqref{eq-s(v)-2} for any $v\in V(H)$. Then: {\rm (1)} $a\geq0$; {\rm (2)} if $a=0$, then $b>0$; if $b=0$, then $a>0$.
\end{proposition}

By Proposition \ref{prop-a-b-fenlei}, one only needs to study the following four cases on $a,b$: (1) $a=0$ and $b>0$; (2) $a=1$ and $b\neq0$; (3) $a\geq2$ and $b\neq0$; (4) $a>0$ and $b=0$. While Cases (1)-(2) and (3) were settled in \cite{Du-4} and \cite{Du-5} respectively, Case (4) remains open.

In this paper, we study the case where the aforementioned $(0,1,2)$-multigraph has a unicyclic $B$-graph and propose several problems.

\section{Preliminaries}


An \emph{$v$-path} is a path with initial vertex $v$. Let $T$ be a tree of order $n$, $P_{l+1}=v_0v_1\cdots v_l$ $(2\leq l\leq n-1)$ be a path of $T$ with length $l$. Now we let $1\leq k\leq \lfloor\frac{l}{2}\rfloor$, $v\in N_T(v_{k})\backslash \{v_{k+1}\}$ and $P_t$ be the longest $v$-path in $T-v_k$. Then we say $v$ is a \emph{$(v_k,t)$-vertex} in $T$.

For example, $v_0$ is a $(v_1,1)$-vertex and $v_3$ is a $(v_4,4)$-vertex in tree $T^*$ (see Figure \ref{vk-t}).

\begin{figure}[!h]
  \centering
  \includegraphics[width=6cm]{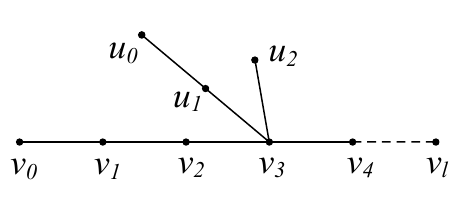}\\
  \caption{The tree $T^*$.}\label{vk-t}
\end{figure}

The following results are useful and interesting.

\begin{lemma}\label{lemma-EV} {\rm (\cite{Du-4})}
Let $T$ be a tree of order $n$, $P_{l+1}=v_0v_1\ldots v_l$ $(2\leq l \leq n-1)$ be a path of $T$ with length $l$, and for any
$v\in\left\{
\begin{array}{ll}
  N_T(v_k)\backslash\{v_{k+1}\}, &\mbox{if $1\leq k\leq\lfloor\frac{l}{2}\rfloor$,}\\
  N_T(v_k)\backslash\{v_{k-1}\}, &\mbox{if $\lfloor\frac{l}{2}\rfloor+1\leq k\leq l-1$},
\end{array}\right.$
$P_t$ be the longest $v$-path of $T-v_k$. Then $P_{l+1}$ is the longest path of $T$ if and only if
$t\leq\left\{
\begin{array}{ll}
  k, &\mbox{if $1\leq k\leq\lfloor\frac{l}{2}\rfloor$,}\\
  l-k, &\mbox{if $\lfloor\frac{l}{2}\rfloor+1\leq k\leq l-1$}.
\end{array}\right.$
\end{lemma}

In a simple graph, a \emph{pendent} vertex is a vertex of degree $1$. We say $v$ is a pendent vertex in $H$ if $v$ is a pendent vertex in the $B$-graph of $H$. If $v_0v_1\ldots v_{l-1}v_l$ is the longest path of $H$, then for any vertex $u\in N(v_1)\backslash\{v_2\}$, $u$ is a pendent vertex, say, $(v_1,1)$-vertex. If $v\in V(H)$ is a $(v_k,t)$-vertex in the $B$-graph of $H$, then we say $v\in V(H)$ is a $(v_k,t)$-vertex in $H$.

\begin{lemma}{\rm (\cite{Du-4})}\label{lem-b=0}
 Let $H,a,b$ be as in Proposition $\ref{prop-a-b-fenlei}$, $u,v\in N(v_1)$ be pendent vertices of $H$. If $w(u,v_1)=1$ and $w(v,v_1)=2$, then $b=0$.
\end{lemma}

\begin{lemma}{\rm (\cite{Du-4})}\label{prop-type}
  Let $H,a,b$ be as in Proposition $\ref{prop-a-b-fenlei}$, $v_1,v_2\in V(H)$ with $v_1\sim v_2$, and for any vertex $v\in N(v_1)\backslash\{v_2\}$, $v$ is a pendent vertex. Then $v_1$ and $v_2$ satisfies one of the following six conditions:\\
  {\rm (1)} $w(v_1,v_2)=1$, $w(v,v_1)=1$ for any vertex $v\in N(v_1)\backslash \{v_2\}$, $d(v_1)=a+b$ $(\geq 2)$ and $d(v_2)=a(a+b-1)+1$; \\
  {\rm (2)} $w(v_1,v_2)=1$, $w(v,v_1)=2$ for any vertex $v\in N(v_1)\backslash \{v_2\}$, $d(v_1)=a+\frac{b}{2}$ $(\geq 3)$ is odd and $d(v_2)=a(a+\frac{b}{2}-2)+2$; \\
  {\rm (3)} $w(v_1,v_2)=2$, $w(v,v_1)=1$ for any vertex $v\in N(v_1)\backslash \{v_2\}$, $d(v_1)=a+b$ $(\geq 3)$ and $d(v_2)=\frac{1}{2}a(a+b-1)+1$; \\
  {\rm (4)} $w(v_1,v_2)=2$, $w(v,v_1)=2$ for any vertex $v\in N(v_1)\backslash \{v_2\}$, $d(v_1)=a+\frac{b}{2}$ $(\geq 4)$ is even and $d(v_2)=\frac{1}{2}a(a+\frac{b}{2}-2)+2$;\\
  {\rm (5)} $w(v_1,v_2)=1$, there exist at least one vertex $v\in N(v_1)\backslash\{v_2\}$ satisfies $w(v,v_1)=1$ and $k_1$ $(\geq1)$ vertices $v'\in N(v_1)\backslash\{v_2\}$ satisfies $w(v',v_1)=2$, $d(v_1)=a$ $(\geq 2k_1+2)$ and $d(v_2)=a^2-a+1-2k_1$;\\
  {\rm (6)} $w(v_1,v_2)=2$, there exist at least one vertex $v\in N(v_1)\backslash\{v_2\}$ satisfies $w(v,v_1)=1$ and $k_2$ $(\geq1)$ vertices $v'\in N(v_1)\backslash\{v_2\}$ satisfies $w(v',v_1)=2$, $d(v_1)=a$ $(\geq 2k_2+3)$ and $d(v_2)=\frac{1}{2}(a^2-a+2-2k_2)$.
\end{lemma}

In the following paper, we denote $d(v_1),d(v_2)$ of $(i)$ in Lemma \ref{prop-type} by $d^{(i)}_{v_1},d^{(i)}_{v_2}$ for $i\in\{1,2,3,4,5,6\}$, respectively.

\begin{lemma}{\rm (\cite{Du-4})}\label{lem-different-1-4}
  Let $H,a,b$ be as in Proposition $\ref{prop-a-b-fenlei}$, $v_2$ be a $(v_3,3)$-vertex, $v_1, v_1'\in N(v_2)\backslash \{v_3\}$ be $(v_2,2)$-vertices in $H$, $v_1$ and $v_2$, $v_1'$ and $v_2$ satisfy  $(i),(j)$ of Lemma $\ref{prop-type}$ for $i,j\in\{1,2,3,4,5,6\}$, respectively. Then $i=j$.
\end{lemma}

\begin{remark}\label{remark-different}
  By the proof of Lemma $\ref{lem-different-1-4}$ in {\rm \cite{Du-4}}, we know that $d_{v_2}^{(i)}\neq d_{v_2}^{(j)}$ holds for any $i,j\in\{1,2,3,4,5,6\}$ and $i\neq j$.
\end{remark}

\begin{lemma}{\rm (\cite{Du-4})}\label{lemma-dv1-dv2}
  Let $H$, $a$, $b$ $(\neq0)$ be as in Proposition $\ref{prop-a-b-fenlei}$, $Q=u_0u_1u_2\ldots u_l$ be a path of $H$ with $l\geq3$. Then we have:\\
  {\rm (1)} if $u_0$ is a $(u_1,1)$-vertex, then $d(u_1)\in\{a+b,a+\frac{b}{2}\}$; \\
  {\rm (2)} if $u_1$ is a $(u_2,2)$-vertex, then $d(u_2)=d_{v_2}^{(i)}$ for some $i\in\{1,2,3,4\}$.
\end{lemma}

\section{Main results}\label{section-main-results}

In this section, we study the case where the aforementioned (0,1,2)-multigraph $H$ that has a unicyclic $B$-graph and exactly two main eigenvalues. Since $H^\circ$ (the $B$-graph of $H$) is a unicyclic graph, we decompose its vertex set as \begin{equation*}
  V(H)=V(H^\circ)=V(C)\cup V(F),
\end{equation*} where $C$ is the unique cycle of $H^\circ$ and $F$ is a forest. Notably, $|V(F)|=0$ is a valid possibility.

Similar to Proposition \ref{prop-a-b-fenlei}, we have the following result.

\begin{proposition}\label{prop-a-b-fenlei-unicyclic}
  Let $H$ be a $(0,1,2)$-multigraph of order $n$ $(\geq3)$ with exactly two distinct main eigenvalues, and suppose its $B$-graph is unicyclic. Let $a,b$ be integers that satisfy \eqref{eq-s(v)-2} for any $v\in V(H)$. Then the structural analysis of $H$ reduces to the following four cases: {\rm (1)} $a=0$ and $b>0$; {\rm (2)} $a=1$ and $b\neq0$; {\rm (3)} $a\geq2$ and $b\neq0$; {\rm (4)} $a>0$ and $b=0$.
\end{proposition}

In Subsections \ref{subsection-F=0} and \ref{subsection-F>0}, we study cases $|V(F)|=0$ and $|V(F)|\neq0$, respectively.

\subsection{$|V(F)|=0$}\label{subsection-F=0}

Let $C_n$ be the cycle of order $n$ $(\geq3)$. Since $C_n$ is $2$-regular and has exactly one main eigenvalue $2$, the following remark is obvious.

\begin{remark}\label{remark-C-n}
  Let $H,a,b$ be as in Proposition $\ref{prop-a-b-fenlei-unicyclic}$. If the $B$-graph of $H$ is $C_n$, then there exist vertices $u,v\in V(H)$ such that $w(u,v)=2$.
\end{remark}

In this subsection, we let $V(H)=V(C)=\{u_1,u_2,\ldots,u_{\frac{n}{2}},v_1,v_2,\ldots,v_{\frac{n}{2}}\}$ and $u_1u_2\cdots u_{\frac{n}{2}}$ $v_{\frac{n}{2}}\cdots v_2$ $v_1u_1$ be a walk of $H$ if $n$ is even, $V(H)=V(C)=\{u_1,u_2,\ldots,u_{\frac{n-1}{2}},v_1,v_2,\ldots,v_{\frac{n-1}{2}},$ $x\}$ and $u_1u_2\cdots u_{\frac{n-1}{2}}$ $xv_{\frac{n-1}{2}}\cdots v_2v_1u_1$ be a walk of $H$ if $n$ is odd.

By Remark \ref{remark-C-n}, we can suppose $w(u_1,v_1)=2$. Then one of the following four cases holds:

(i) $w(u_1,u_2)=1$ and $w(v_1,v_2)=1$;

(ii) $w(u_1,u_2)=1$ and $w(v_1,v_2)=2$;

(iii) $w(u_1,u_2)=2$ and $w(v_1,v_2)=1$;

(iv) $w(u_1,u_2)=2$ and $w(v_1,v_2)=2$.

\begin{lemma}\label{lemma-symmetric}
  Let $H,a,b$ be as in Proposition $\ref{prop-a-b-fenlei-unicyclic}$, where $V(H)=V(C)$. If $w(u_1,u_2)=w(v_1,v_2)$. Then $w(u_i,u_{i+1})=w(v_i,v_{i+1})$ for $2\leq i\leq \lfloor\frac{n}{2}\rfloor-1$. Moreover, $w(u_{\frac{n-1}{2}},x)=w(v_{\frac{n-1}{2}},x)$ if $n$ is odd.
\end{lemma}

\begin{proof}
  It is clear that $d(u_1)=d(v_1)$. We prove this by inductive hypothesis on $i$ $(2\leq i\leq \lfloor\frac{n}{2}\rfloor-1)$.

  By \eqref{eq-s(v)-2}, we have
  \begin{equation*}
    \begin{cases}
      a\cdot d(u_1)+b=s(u_1)=2d(v_1)+w(u_1,u_2)d(u_2)\\
      a\cdot d(v_1)+b=s(v_1)=2d(u_1)+w(v_1,v_2)d(v_2).
    \end{cases}
  \end{equation*} Then $d(u_2)=d(v_2)$ by $d(u_1)=d(v_1)$ and $w(u_1,u_2)=w(v_1,v_2)$. Clearly, $w(u_2,u_3)=w(v_2,v_3)$ by $d(u_2)=d(v_2)$ and $w(u_1,u_2)=w(v_1,v_2)$. Then the result holds for $i=2$.

  Suppose that the result holds for $i\leq k$ $(2\leq k\leq\lfloor\frac{n}{2}\rfloor-2)$. By induction hypothesis, we have $w(u_{k-2},u_{k-1})=w(v_{k-2},v_{k-1})$, $w(u_{k-1},u_k)=w(v_{k-1},v_k)$ and $w(u_k,u_{k+1})=w(v_k,v_{k+1})$, thus $d(u_{k-1})=d(v_{k-1})$ and $d(u_k)=d(v_k)$. Then by \eqref{eq-s(v)-2},  we have \begin{equation*}
    \begin{cases}
      a\cdot d(u_k)+b=s(u_k)=w(u_{k-1},u_k)d(u_{k-1})+w(u_k,u_{k+1})d(u_{k+1}),\\
      a\cdot d(v_k)+b=s(v_k)=w(v_{k-1},v_k)d(v_{k-1})+w(v_k,v_{k+1})d(v_{k+1}),
    \end{cases}
  \end{equation*} which implies $d(u_{k+1})=d(v_{k+1})$. Combining with $w(u_k,u_{k+1})=w(v_k,v_{k+1})$, we have $w(u_{k+1},u_{k+2})=w(v_{k+1},v_{k+2})$. Then the result holds for $i=k+1$.

  According the above discussion, it is not difficult to see that $w(u_{\frac{n-1}{2}},x)=w(v_{\frac{n-1}{2}},x)$ if $n$ is odd, then we complete the proof.
\end{proof}

\vskip0.2cm
In Subsections \ref{subsection-Cn-1}, \ref{subsection-Cn-2}, \ref{subsection-Cn-3} and \ref{subsection-Cn-4}, we characterize $H$ with $B$-graph being $C_n$, corresponding to case (i), (ii), (iii) and (iv), respectively.

Now, we define seven types of $(0,1,2)$-multigraphs, denoted as $U^1_{4t},U^2_{3t},U^3_{3t},U^4_{5t},U^5_{4t},U^6_{2t},$ $U^7_{t}$ (see Figure \ref{U1-U7}), where $t\geq1$. We use $U=[n_1,n_2]$ to denote a multigraph $U$ that is formed by the cyclic connection of subgraphs $U'$. Here, $n_1\geq1$ and $n_2\geq0$ are integers, and the subgraph $U'$ is constructed by sequentially connecting $n_1$ consecutive parallel edges and $n_2$ consecutive single edges. The term ``cyclic connection'' here means repeating the connection of the subgraph $U'$ in a certain order to form the complete multigraph $U$. All seven types of multigraphs follow the aforementioned generation method of cyclic connection of parallel edges and single edges. In fact, $U^1_{4t}=[1,3]$, $U^2_{3t}=[1,2]$, $U^3_{3t}=[2,1]$, $U^4_{5t}=[3,2]$, $U^5_{4t}=[3,1]$, $U^6_{2t}=[1,1]$ and $U^7_{t}=[1,0]$. It is clear that $t$ is the number of $U'$ in $U$, and thus $|V(U)|=(n_1+n_2)t$.

\begin{proposition}\label{prop-U1-U7}
  Let $U^1_{4t},U^2_{3t},U^3_{3t},U^4_{5t},U^5_{4t},U^6_{2t}$ and $U^7_{t}$ $(t\geq1)$ be defined as above. Then $U^1_{4t},U^2_{3t},U^3_{3t},U^4_{5t},U^5_{4t}$ have exactly two main eigenvalues, $U^6_{2t},U^7_{t}$ have exactly one main eigenvalues.
\end{proposition}

\begin{proof}
  We observe that $U^6_{2t}$ and $U^7_{t}$ are regular, thus they have exactly one main eigenvalues.  For $U^1_{4t},U^2_{3t},U^3_{3t},U^4_{5t}$ and $U^5_{4t}$, it easy to verify that $a\cdot d(v)+b=s(v)$ for any vertex $v\in U^i$ $(1\leq i\leq 5)$, with $a,b$ values in Table \ref{table-U1-U5}. Thus, $U^1_{4t},U^2_{3t},U^3_{3t},U^4_{5t}$ and $U^5_{4t}$ have exactly two main eigenvalues by Theorem \ref{lemma-s(v)}.
\end{proof}

\begin{table}[h]
\renewcommand{\arraystretch}{1.5}\caption{The values of $a,b$.}\label{table-U1-U5}\vskip0.05cm
\centering
{
\begin{tabular}{cccccc}\hline\hline
$Graph$& $U^1_{4t}$ &$U^2_{3t}$ &$U^3_{3t}$ &$U^4_{5t}$ &$U^5_{4t}$ \\
\hline\hline
$a$& 3&2&1&4&3\\
\hline
$b$& -1&2&8&-2&2\\
\hline\hline
\end{tabular}}
\end{table}

\begin{figure}[!h]
	\centering
	\begin{tikzpicture}
		\node[anchor=south west,inner sep=0] (image) at (0,0) {\includegraphics[width=0.9\textwidth]{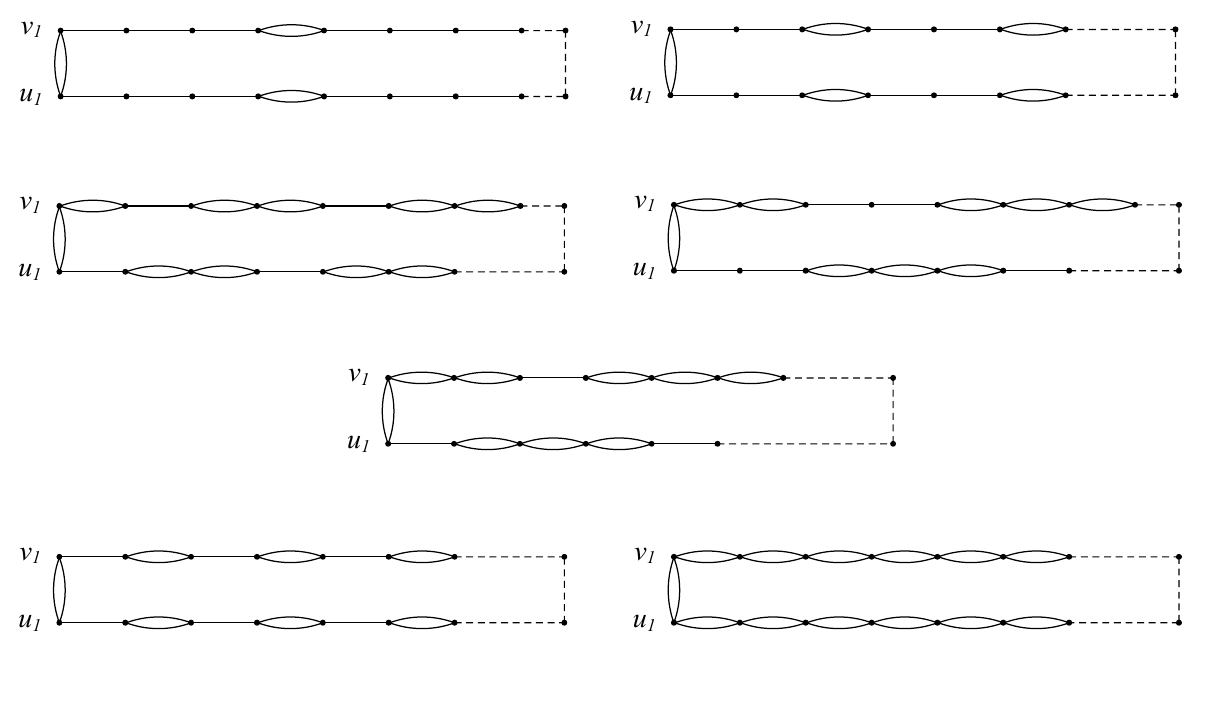}};
		\begin{scope}[
			x={(image.south east)},
			y={(image.north west)}
			]
			\node [black, font=\bfseries] at (0.25,0.8) {$U^1_{4t}$};
            \node [black, font=\bfseries] at (0.75,0.8) {$U^2_{3t}$};
            \node [black, font=\bfseries] at (0.25,0.55) {$U^3_{3t}$};
			\node [black, font=\bfseries] at (0.75,0.55) {$U^4_{5t}$};
            \node [black, font=\bfseries] at (0.5,0.3) {$U^5_{4t}$};
            \node [black, font=\bfseries] at (0.25,0.05) {$U^6_{2t}$};
            \node [black, font=\bfseries] at (0.75,0.05) {$U^7_{t}$};
		\end{scope}
	\end{tikzpicture}
	\caption{The graphs $U^1_{4t},U^2_{3t},U^3_{3t},U^4_{5t},U^5_{4t},U^6_{2t},U^7_{t}$ with $t\geq1$.}\label{U1-U7}
\end{figure}
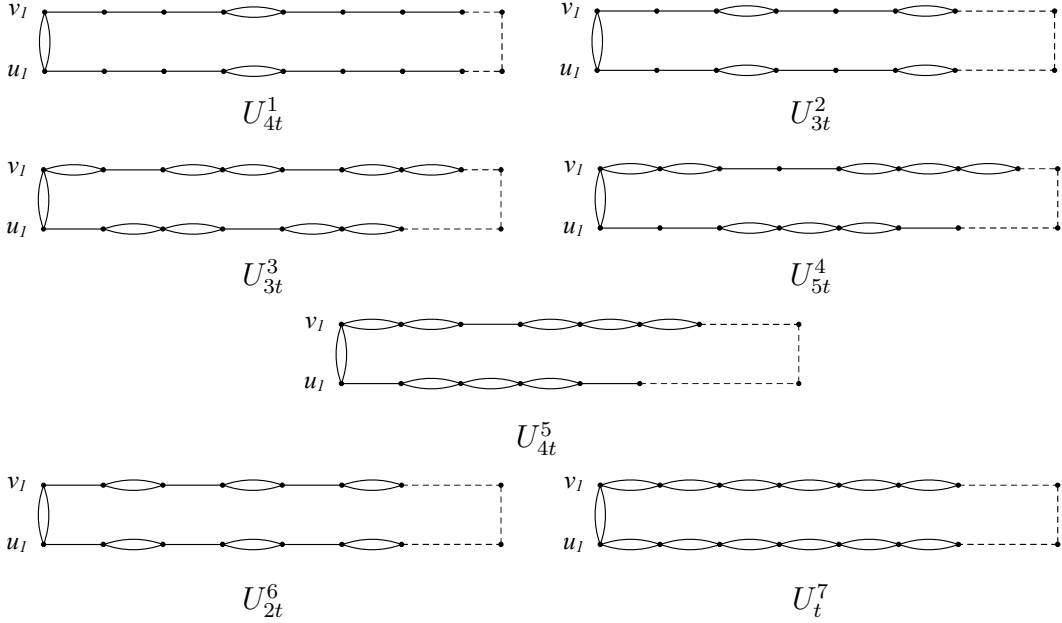

\subsubsection{$w(u_1,u_2)=1$ and $w(v_1,v_2)=1$}\label{subsection-Cn-1}

\begin{theorem}
  Let $H,a,b$ be as in Proposition $\ref{prop-a-b-fenlei-unicyclic}$, where $V(H)=V(C)$. If $w(u_1,u_2)=1$ and $w(v_1,v_2)=1$, then $H\in \{U^1_{4t},U^2_{3t}\}$ $(t\geq1)$.
\end{theorem}

\begin{proof}
  It is clear that $d(v_1)=d(u_1)=3$. By \eqref{eq-s(v)-2}, we have
  \begin{equation}\label{eq-type1}
  \begin{cases}
    ad(v_1)+b=s(v_1)=w(u_1,v_1)d(u_1)+w(v_1,v_2)d(v_2),\\ ad(v_2)+b=s(v_2)=w(v_1,v_2)d(v_1)+w(v_2,v_3)d(v_3).
  \end{cases}
  \end{equation} Then we complete the proof by the following two cases.

  \textbf{Case 1.} $w(v_2,v_3)=1$.

  Now $d(v_2)=2$. Then by \eqref{eq-type1}, we have $3a+b=s(v_1)=2d(u_1)+d(v_2)=8$ and $2a+b=s(v_2)=3+d(v_3)$, which implies $a=5-d(v_3)$. It is clear that $d(v_3)=2$ if $w(v_3,v_4)=1$ and $d(v_3)=3$ if $w(v_3,v_4)=2$.

  \textbf{Subcase 1.1.} $w(v_3,v_4)=1$.

  Now $a=3$, and then $b=-1$ by $3a+b=8$. By \eqref{eq-s(v)-2}, we have $a\cdot d(v_3)+b=s(v_3)=d(v_2)+d(v_4)$, which implies $d(v_4)=3$. Then $w(v_4,v_5)=2$ by $d(v_4)=3$ and $w(v_3,v_4)=1$.

  By applying \eqref{eq-s(v)-2} to $s(v_4)$, $s(v_5)$, $s(v_6)$, $s(v_7)$ and $s(v_8)$ in turn, we can find that $H[v_1,v_2,v_3,v_4,v_5]\cong H[v_5,v_6,v_7,v_8,v_9]$. And so on, it is clear that $H[v_5,v_6,v_7,v_8,v_9]\cong H[v_9,v_{10},v_{11},v_{12},v_{13}]\cong \cdots$. Combining with Lemma \ref{lemma-symmetric}, we conclude that $H$ is generated by repeatedly extending $H[v_1,v_2,v_3,v_4,v_5]$. Thus $n=|V(H)|=|E(H)|=|E(H[v_1,v_2,v_3,v_4,v_5])|t=4t$ $(t\geq 1)$ is even.

  By above discussion and the structure of $H$, it is easy to check that  $H\cong U^1_{4t}$ ($t\geq1$).

 \textbf{Subcase 1.2.} $w(v_3,v_4)=2$.

  Now $a=2$, and then $b=2$ by $3a+b=8$. By \eqref{eq-s(v)-2}, we have $a\cdot d(v_3)+b=s(v_3)=d(v_2)+2d(v_4)$, which implies $d(v_4)=3$. Then $w(v_4,v_5)=1$ by $d(v_4)=3$ and $w(v_3,v_4)=2$.

  By applying \eqref{eq-s(v)-2} to $s(v_4)$, $s(v_5)$, $s(v_6)$ and $s(v_7)$ in turn, we can find that $H[v_1,v_2,v_3,v_4]$ $\cong H[v_4,v_5,v_6,v_7]$. And so on, it is clear that $H[v_4,v_5,v_6,v_7]\cong H[v_7,v_8,v_9,v_{10}]\cong \cdots$. Combining with Lemma \ref{lemma-symmetric}, we conclude that $H$ is generated by repeatedly extending $H[v_1,v_2,v_3,v_4]$. Thus, $n=|E(H[v_1,v_2,v_3,v_4])|t=3t$ is odd and $H\cong U^2_{3t}$ ($t\geq1$).

  \textbf{Case 2.} $w(v_2,v_3)=2$.

  Now $d(v_2)=3$. Then by \eqref{eq-type1}, we have $3a+b=s(v_1)=2d(u_1)+d(v_2)=9$ and $3a+b=s(v_2)=3+2d(v_3)$, which implies $d(v_3)=3$. Then $w(v_3,v_4)=1$ by $d(v_3)=3$ and $w(v_2,v_3)=2$. By applying \eqref{eq-s(v)-2} to $s(v_3)$ and $s(v_4)$ in turn, we can find that $H[v_1,v_2,v_3]\cong H[v_3,v_4,v_5]$. And so on, it is clear that $H[v_3,v_4,v_5]\cong H[v_5,v_6,v_7]\cong \cdots$. Combining with Lemma \ref{lemma-symmetric}, we conclude that $H$ is generated by repeatedly extending $H[v_1,v_2,v_3]$. Thus, $n=|E(H[v_1,v_2,v_3])|t=2t$ is even and $H\cong U^6_{2t}$ ($t\geq1$). However, $U^6_{2t}$ is $3$-regular and has exactly one main eigenvalue, a contradiction.
\end{proof}

\subsubsection{$w(u_1,u_2)=1$ and $w(v_1,v_2)=2$}\label{subsection-Cn-2}

\begin{theorem}\label{theorem-cycle-2type}
  Let $H,a,b$ be as in Proposition $\ref{prop-a-b-fenlei-unicyclic}$, where $V(H)=V(C)$. If $w(u_1,u_2)=1$ and $w(v_1,v_2)=2$, then $H\in\{U^3_{3t},U^4_{5t},U^5_{4t}\}$ $(t\geq1)$.
\end{theorem}

\begin{proof}
  It clear that $d(u_1)=3$ and $d(v_1)=4$. By \eqref{eq-s(v)-2}, we have
  \begin{equation}\label{eq-type2}
  \begin{cases}
    ad(u_1)+b=s(u_1)=w(u_1,v_1)d(v_1)+w(u_1,u_2)d(u_2),\\ ad(v_1)+b=s(v_1)=w(v_1,u_1)d(u_1)+w(v_1,v_2)d(v_2).
  \end{cases}
  \end{equation} Then we complete the proof by the following two cases.

  \textbf{Case 1.} $w(v_2,v_3)=1$.

  Now $d(v_2)=3$. By \eqref{eq-type2}, we have $3a+b=8+d(u_2)$ and $4a+b=12$, which implies $a=4-d(u_2)$. Since $w(u_1,u_2)=1$, we have $d(u_2)\in\{2,3\}$, and thus $a\in\{1,2\}$.

  By \eqref{eq-s(v)-2}, we have $a\cdot d(v_2)+b=s(v_2)=2d(v_1)+d(v_3)$, which implies $3a+b=8+d(v_3)$. Thus, $d(u_2)=d(v_3)$.

  If $d(u_2)=d(v_3)=2$, then $a=2$, and further $b=4$ by $3a+b=8+d(v_3)$. It is also known that $w(v_3,v_4)=1$ by $w(v_2,v_3)=1$ and $d(v_3)=2$. Then by \eqref{eq-s(v)-2}, we have $ad(v_3)+b=s(v_3)=w(v_2,v_3)d(v_2)+w(v_3,v_4)d(v_4)$, which implies $d(v_4)=5$. However, it is clear that $d(v)\leq4$ for any $v\in V(H)$ since the $B$-graph of $H$ is $C_n$, a contradiction.

  If $d(u_2)=d(v_3)=3$, then it is easy to check that $a=1$, $b=8$ and  $w(v_3,v_4)=2$. By applying \eqref{eq-s(v)-2} to $s(v_3)$, $s(v_4)$ and $s(v_5)$ in turn, we can find that $H[u_1,v_1,v_2,v_3]\cong H[v_3,v_4,v_5,v_6]$. And so on, it is clear that $H[v_3,v_4,v_5,v_6]\cong H[v_6,v_7,v_8,v_9]\cong \cdots$. On the other hand, by $H[u_1,u_2,u_3]\cong H[v_2,v_3,v_4]$ and a similar discussion, we can deduce that $H[v_3,v_4,v_5,v_6]\cong H[u_2,u_3,u_4,u_5]\cong H[u_5,u_6,u_7,u_8]\cong\cdots$. Thus, $H$ is generated by repeatedly extending $H[u_1,v_1,v_2,v_3]$, and $n=|E(H[u_1,v_1,v_2,v_3])|t=3t$ $(t\geq1)$ is odd. Therefore, $H\cong U^3_{3t}$.

  \textbf{Case 2.} $w(v_2,v_3)=2$.

  Now $d(v_2)=4$. By \eqref{eq-type2}, we have $3a+b=8+d(u_2)$ and $4a+b=14$, which implies $a=6-d(u_2)$. Since $w(u_1,u_2)=1$, we have $d(u_2)\in\{2,3\}$, and thus $a\in\{3,4\}$.

  By \eqref{eq-s(v)-2}, we have $4a+b=s(v_2)=w(v_1,v_2)d(v_1)+w(v_2,v_3)d(v_3)$, which implies $d(v_3)=3$ by $4a+b=14$. Thus, $w(v_3,v_4)=1$.

  \textbf{Subcase 2.1.} $w(u_2,u_3)=1$.

  Now $d(u_2)=2$, $a=4$ by $a=6-d(u_2)$, and further $b=-2$ by $4a+b=14$. By \eqref{eq-s(v)-2}, we have $3a+b=s(v_3)=w(v_2,v_3)d(v_2)+w(v_3,v_4)d(v_4)$, which implies $d(v_4)=2$. Then $w(v_3,v_4)=w(v_4,v_5)=1$. Similarly, we can deduce that $d(v_5)=3$ by $2a+b=s(v_4)=w(v_3,v_4)d(v_3)+w(v_4,v_5)d(v_5)$. Then $w(v_5,v_6)=2$. By applying \eqref{eq-s(v)-2} to $s(v_5)$, $s(v_6)$, $s(v_7)$, $s(v_8)$ and $s(v_9)$ in turn, we can find that $H[u_1,v_1,v_2,v_3,v_4,v_5]\cong H[v_5,v_6,v_7,v_8,v_9,v_{10}]$. And so on, it is clear that $H[v_5,v_6,v_7,v_8,v_9,v_{10}]\cong H[v_{10},v_{11},v_{12},$ $v_{13},v_{14},v_{15}]\cong \cdots$. On the other hand, by $H[v_1,u_1,u_2,u_3]\cong H[v_2,v_3,v_4,v_5]$ and a similar discussion, we can deduce that $H[v_3,v_4,v_5,v_6,v_7,v_8]\cong H[u_1,u_2,u_3,u_4,u_5,u_6]\cong H[u_6,u_7,u_8,u_9,u_{10},u_{11}]\cong\cdots$. Thus, $H$ is generated by repeatedly extending $H[u_1,v_1,v_2,v_3,v_4,v_5]$, and $n=|E(H[u_1,$ $v_1,v_2,v_3,v_4,v_5])|t=5t$ $(t\geq1)$ is odd. Therefore, $H\cong U^4_{5t}$.

  \textbf{Subcase 2.2.} $w(u_2,u_3)=2$.

  Now $d(u_2)=3$, $a=3$ by $a=6-d(u_2)$, and further $b=2$ by $4a+b=14$. By \eqref{eq-s(v)-2}, we have $3a+b=s(v_3)=w(v_2,v_3)d(v_2)+w(v_3,v_4)d(v_4)$, which implies $d(v_4)=3$, and thus $w(v_4,v_5)=2$.

  By applying \eqref{eq-s(v)-2} to $s(v_4)$, $s(v_5)$, $s(v_6)$ and $s(v_7)$ in turn, we can find that $H[u_1,v_1,v_2,$ $v_3,v_4]\cong H[v_4,v_5,v_6,v_7,v_8]$. And so on, it is clear that $H[v_4,v_5,v_6,v_7,v_8]\cong H[v_8,v_9,v_{10},$ $v_{11},v_{12}]\cong \cdots$. On the other hand, by $H[v_1,u_1,u_2,u_3]\cong H[v_2,v_3,v_4,v_5]$ and a similar discussion, we can deduce that $H[v_3,v_4,v_5,v_6,v_7]\cong H[u_1,u_2,u_3,u_4,u_5]\cong H[u_5,u_6,u_7,u_8,u_9]$ $\cong \cdots$. Thus, $H$ is generated by repeatedly extending $H[u_1,v_1,v_2,v_3,v_4]$, and $n=|E(H[u_1,v_1,v_2,v_3,v_4])|t=4t$ $(t\geq1)$ is even. Therefore, $H\cong U^5_{4t}$.
\end{proof}

\subsubsection{$w(u_1,u_2)=2$ and $w(v_1,v_2)=1$}\label{subsection-Cn-3}
Similar to Subsection \ref{subsection-Cn-2} and we omit it.

\subsubsection{$w(u_1,u_2)=w(v_1,v_2)=2$}\label{subsection-Cn-4}

\begin{theorem}
  Let $H,a,b$ be as in Proposition $\ref{prop-a-b-fenlei-unicyclic}$, where $V(H)=V(C)$. If $w(u_1,u_2)=w(v_1,v_2)=2$, then $H\in\{U^4_{5t},U^5_{4t}\}$ $(t\geq1)$.
\end{theorem}

\begin{proof}
  It is clear that $d(u_1)=d(v_1)=4$. By \eqref{eq-s(v)-2}, we have
  \begin{equation}\label{eq-type3}
  \begin{cases}
    ad(v_1)+b=s(v_1)=w(u_1,v_1)d(u_1)+w(v_1,v_2)d(v_2),\\ ad(v_2)+b=s(v_2)=w(v_1,v_2)d(v_1)+w(v_2,v_3)d(v_3).
  \end{cases}
  \end{equation}
  Then we complete the proof by the following two cases.

  \textbf{Case 1.} $w(v_2,v_3)=1$.

  Now $d(v_2)=3$. By \eqref{eq-type3}, we have $4a+b=14$ and $3a+b=8+d(v_3)$, and further $a=6-d(v_3)$. By $w(v_2,v_3)=1$, we have $d(v_3)\in\{2,3\}$.

  \textbf{Subcase 1.1.} $d(v_3)=2$.

  Now $w(v_3,v_4)=w(v_2,v_3)=1$, $a=4$ by $a=6-d(v_3)$, and further $b=-2$ by $4a+b=14$. Through a discussion similar to Subcase 2.1 of Theorem \ref{theorem-cycle-2type}, we have $H\cong U^4_{5t}$ $(t\geq1)$.

  \textbf{Subcase 1.2.} $d(v_3)=3$.

  Now $w(v_3,v_4)=2$, $a=3$ by $a=6-d(v_3)$, and further $b=2$ by $4a+b=14$. Through a discussion similar to Subcase 2.2 of Theorem \ref{theorem-cycle-2type}, we have $H\cong U^5_{4t}$ $(t\geq1)$.

  \textbf{Case 2.} $w(v_2,v_3)=2$.

  Now $d(v_2)=4$. By \eqref{eq-type3}, we have $4a+b=16$ and $4a+b=8+2d(v_3)$, which implies $d(v_3)=4$. Thus, $w(v_3,v_4)=2$. Next we show that $d(v_k)=4$ by inductive hypothesis on $k$ $(3\leq k\leq\lfloor\frac{n}{2}\rfloor-1)$. Suppose that the result holds for $v_k$ $(k\leq \lfloor\frac{n}{2}\rfloor-1)$, then we have $d(v_k)=4$ and $w(v_{k-1},v_k)=w(v_k,v_{k+1})=2$. By \eqref{eq-s(v)-2}, we have $16=4a+b=s(v_k)=2d(v_{k-1})+2d(v_{k+1})$, which implies $d(v_{k+1})=4$. On the other hand, we can deduce that $d(u_i)=4$ for $1\leq i\leq\lfloor\frac{n}{2}\rfloor$ by Lemma \ref{lemma-symmetric} and $w(u_1,u_2)=w(v_1,v_2)$. It is easy to check that $w(u_{\lfloor\frac{n}{2}\rfloor},v_{\lfloor\frac{n}{2}\rfloor})=2$ if $n$ is even and $w(u_{\lfloor\frac{n}{2}\rfloor},x)=w(x,v_{\lfloor\frac{n}{2}\rfloor})=2$ if $n$ is odd. Thus, $H\cong U^7_t$ $(t\geq1)$. However, $U^7_t$ is $4$-regular and has exactly one main eigenvalue, a contradiction.
\end{proof}

\subsection{$|V(F)|\neq0$}\label{subsection-F>0}

Let $G_{uv}G'$ be the graph obtained from $G\cup G'$ by adding an edge joining the vertex $u$ of $G$ to the vertex $v$ of $G'$. If $G'$ is a tree, then we say $G'$ is a \emph{pendent tree} of $G_{uv}G'$.

In this subsection, we let $H^\circ\cong H'_{uv}T$, $v=v_l$, $u=v_{l+1}$ and $Q=v_0v_1\cdots v_{l-1}v_l$ $(l\geq0)$ be the longest $v$-path in $T$. It is obvious that $H'$ has exactly one induced subgraph that is a cycle, which we denote as $C$.

\begin{remark}\label{remark-main}
  In \cite{Du-4,Du-5}, the authors considered that the $B$-graph of $H$ being a tree, and consistently assumed that $P=v_0v_1\cdots v_{l'}$ represents the longest path of $H$ in their propositions, lemmas and theorems {\rm(}for example, Proposition \ref{prop-a-b-fenlei} here{\rm )}. It is known from Lemma \ref{lemma-EV} that this assumption implies the following two conclusions:

  {\rm (I)} any vertex $v\in N(v_i)\backslash\{v_{i+1}\}$ $(1\leq i\leq \lfloor\frac{l'}{2}\rfloor)$ is a $(v_i,t)$-vertex for some $t\in\{1,2,\ldots,i\}$;

  {\rm (II)} any vertex $v\in N(v_i)\backslash\{v_{i-1}\}$ $(\lfloor\frac{l'}{2}\rfloor+1\leq i\leq l'-1)$ is a $(v_i,t)$-vertex for some $t\in\{1,2,\ldots,l'-i\}$.

  However, due to the symmetry of the graph, only conclusion {\rm (I)} is actually required in the proof of the main results in \cite{Du-4,Du-5}.

  On the other hand, we defined $Q=v_0v_1\cdots v_{l-1}v_l$ $(l\geq0)$ as the longest $v$-path in $T$, where $H^\circ\cong H'_{uv}T$, $v=v_l$ and $u=v_{l+1}$. Then it is clear that any vertex $v'\in N(v_k)\backslash\{v_{k+1}\}$ $(1\leq k\leq l)$ is a $(v_k,t)$-vertex for some $t\in\{1,2,\ldots,k\}$.

  Therefore, $Q$ and the first half of $P$ share the same preconditions {\rm (}i.e., conclusion {\rm (I)}{\rm )}. Thus, we can conclude that all the results regarding $P$ in {\rm \cite{Du-4,Du-5}} are also applicable to $Q$ here.
\end{remark}

\subsubsection{$a=0$ and $b\neq0$}

\begin{lemma}\label{lemma-Hi}{\rm (\cite{Du-4})}
  Let $H, a, b$ be as in Proposition \ref{prop-a-b-fenlei}, $P=v_0v_1\ldots v_l$ be the longest path of $H$ with $l\geq3$. If $a=0$, then for all $k$ $(0\leq k\leq \lfloor\frac{l}{2}\rfloor)$, we have $
  d(v_k)=\left\{
    \begin{array}{ll}
      2, &\mbox{k is even},\\
      \frac{b}{2}, &\mbox{k is odd},
    \end{array}\right.
  $ $w(v_0,v_1)=2$, $w(v_k,v_{k+1})=1$ for any $k\geq1$, and $v$ is a $(v_k,t)$-vertex for any vertex $v\in N(v_k)\backslash\{v_{k+1}\}$ with $t\leq k$, where $k$ $(\geq1)$ and $t$ are odd. Moreover, for any vertex $v\in N(v_k)\backslash\{v_{k-1},v_{k+1}\}$ with $k$ is odd and any path $P'=u_0u_1\ldots u_{t-2}u_{t-1}u_t$ with $u_{t-1}=v$, $u_t=v_k$, we have $
  d(u_r)=\left\{
    \begin{array}{ll}
      2, &\mbox{r is even}\\
      \frac{b}{2}, &\mbox{r is odd}
    \end{array}\right.
  $ for $0\leq r\leq t-1$, $w(u_0,u_1)=2$, $w(u_{r},u_{r+1})=1$ for $1\leq r\leq t-1$, and for any vertex $u\in N(u_r)\backslash\{u_{r+1}\}$, $u$ is a $(u_r,s)$-vertex with $s\leq r$, where $r$ $(1\leq r\leq t-1)$ and $s$ are odd.
\end{lemma}

\begin{lemma}\label{lemma-a=0}
  Let $H, a, b$ be as in Proposition \ref{prop-a-b-fenlei-unicyclic}, $H^\circ\cong H'_{uv}T$, $v=v_l$, $u=v_{l+1}$, $Q=v_0v_1\cdots v_{l-1}v_l$ $(l\geq0)$ be the longest $v$-path in $T$. If $a=0$, then $l$ is even, $d(v_i)=\left\{
    \begin{array}{ll}
      2, &\mbox{$i$ is even},\\
      \frac{b}{2}, &\mbox{$i$ is odd},
    \end{array}\right.$ $w(v_0,v_1)=2$, $w(v_i,v_{i+1})=1$ for $1\leq i\leq l$, and $v'$ is a $(v_k,t)$-vertex for any vertex $v'\in N(v_k)\backslash\{v_{k+1}\}$ with $t\leq k$, where $k$ $(1\leq k\leq l+1)$ and $t$ are odd,  moreover, for any vertex $v'\in N(v_k)\backslash\{v_{k-1},v_{k+1}\}$ with $k$ is odd and any path $Q'=u_0u_1\ldots u_{t-2}u_{t-1}u_t$ with $u_{t-1}=v'$, $u_t=v_k$, we have $
  d(u_r)=\left\{
    \begin{array}{ll}
      2, &\mbox{r is even}\\
      \frac{b}{2}, &\mbox{r is odd}
    \end{array}\right.
  $ for $0\leq r\leq t-1$, $w(u_0,u_1)=2$, $w(u_{r},u_{r+1})=1$ for $1\leq r\leq t-1$, and for any vertex $u\in N(u_r)\backslash\{u_{r+1}\}$, $u$ is a $(u_r,s)$-vertex with $s\leq r$, where $r$ $(1\leq r\leq t-1)$ and $s$ are odd.
\end{lemma}

\begin{proof}
  We complete the proof by the following two cases.

  \textbf{Case 1.} $l=0$.

  If $l=0$, then $V(T)=\{v\}$. Next we show $w(u,v)=2$.
  Suppose that $w(u,v)=1$, then $d(v)=1$. By \eqref{eq-s(v)-2}, we have $b=s(v)=w(u,v)d(u)$, which implies $d(u)=b$. By \eqref{eq-s(v)-2}, we have $b=s(u)=w(u,v)d(v)+\sum\limits_{u'\sim u, u'\neq v}w(u,u')d(u')$, where $w(u,v)+\sum\limits_{u'\sim u,u'\neq v}w(u,u')=d(u)$. Combining with $d(u)=b$, it is easy to check that $w(u,u')=d(u')=1$ for any $u'\in N(u)\backslash\{v\}$. However, there exists at least one vertex $u'\in V(C)$ with $d(u')\geq2$, a contradiction.

  If $w(u,v)=2$, then by \eqref{eq-s(v)-2}, we have $b=s(v)=w(u,v)d(u)$, which implies $d(u)=\frac{b}{2}$.

  \textbf{Case 2.} $l\geq1$.

  Let $V(H^\circ)=V(C)\cup V(F)$. Since $v\in N(u)$ and there exist two vertices $u_1,u_2\in N(u)\cap V(C)$, we have $d(u)\geq3$. On the other hand, we know that $d(u)=d(v_{l+1})\in\{2,\frac{b}{2}\}$ by Remark \ref{remark-main} and Lemma \ref{lemma-Hi}. Thus, $l$ is even by $\frac{b}{2}\geq3$. The rest part of the result is obvious.
\end{proof}
\vskip0.2cm

Now we define $\mathcal{H}^1(b,t)$ ($b=4k+2\geq 6$ and $t\geq4$ are even) be a set of (0,1,2)-multigraphs such that for any $H\in \mathcal{H}(b)$, $H$ satisfies the following three conditions:

(1) $H^\circ$ is a unicyclic graph;

(2) If we let $V(H^\circ)=V(C)\cup V(F)$, then $w(u,v)=1$ for any pair of adjacency vertices $u,v\in V(C)$. Moreover, the degrees of vertices of $C$ are either $2$ or $\frac{b}{2}$ $(\geq3)$, and they appear alternately;

(3) Let $u,u_1,u_2\in V(C)$, $d(u)=\frac{b}{2}$, $N(u)=\{u_1,u_2,p_1,p_2,\ldots,p_t\}$ $(t\geq 1)$, and $H^\circ\cong {H^i}_{up_i}T_i$ for $1\leq i\leq t$. If $Q=v_0v_1\cdots v_{l-1}v_l$ $(v_l=p_i)$ is the longest $p_i$-path in $T_i$, then $l$ is even, $w(v_0,v_1)=2$, $w(v_l,u)=w(v_i,v_{i+1})=1$ for $1\leq i\leq l-1$, $d(v_i)=2$ if $i$ is even and $d(v_i)=\frac{b}{2}$ otherwise. Moreover, if there exists another $v_i$-path $u_0u_1\cdots u_{m-2}u_{m-1}v_i$ $(u_m=v_i)$ with $u_0$ as its pendent vertex, then $m$ $(\leq i)$ is odd, $d(u_j)=2$ if $j$ is even and $d(u_j)=\frac{b}{2}$ if $j$ is odd, $w(u_0,u_1)=2$, $w(u_{m-1},v_i)=w(u_j,u_{j+1})=1$ for $1\leq j\leq m-2$, and for any vertex $u\in N(u_r)\backslash\{u_{r+1}\}$, $u$ is a $(u_r,s)$-vertex, where $r$ $(1\leq r\leq m-1)$ and $s$ $(\leq r)$ are odd.

For example, $H_1,H_2,H_3,H_4\in \mathcal{H}^1(6,4)$ and $H_5\in\mathcal{H}^1(10,6)$ (see Figure \ref{H1-H5}).

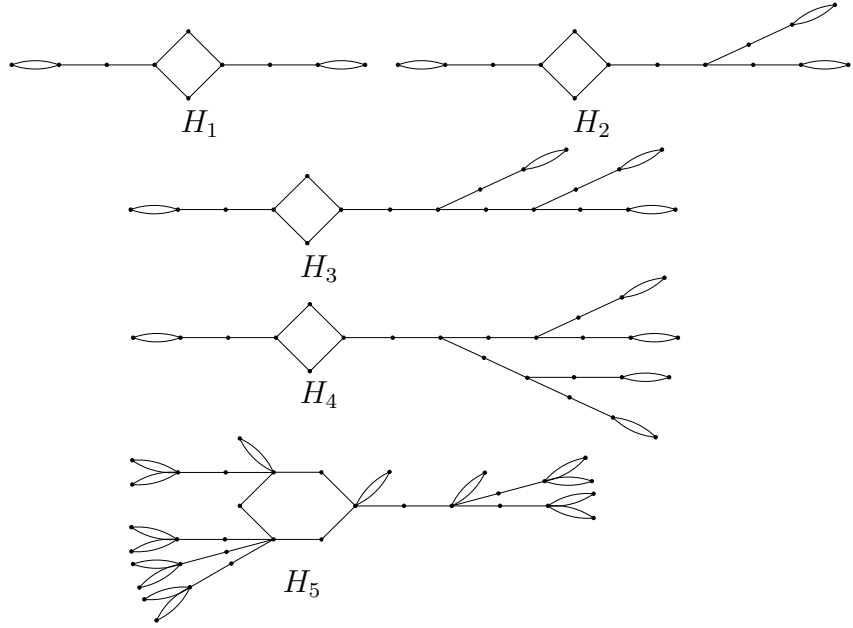
\begin{figure}[!h]
	\centering
	\begin{tikzpicture}
		\node[anchor=south west,inner sep=0] (image) at (0,0) {\includegraphics[width=0.7\textwidth]{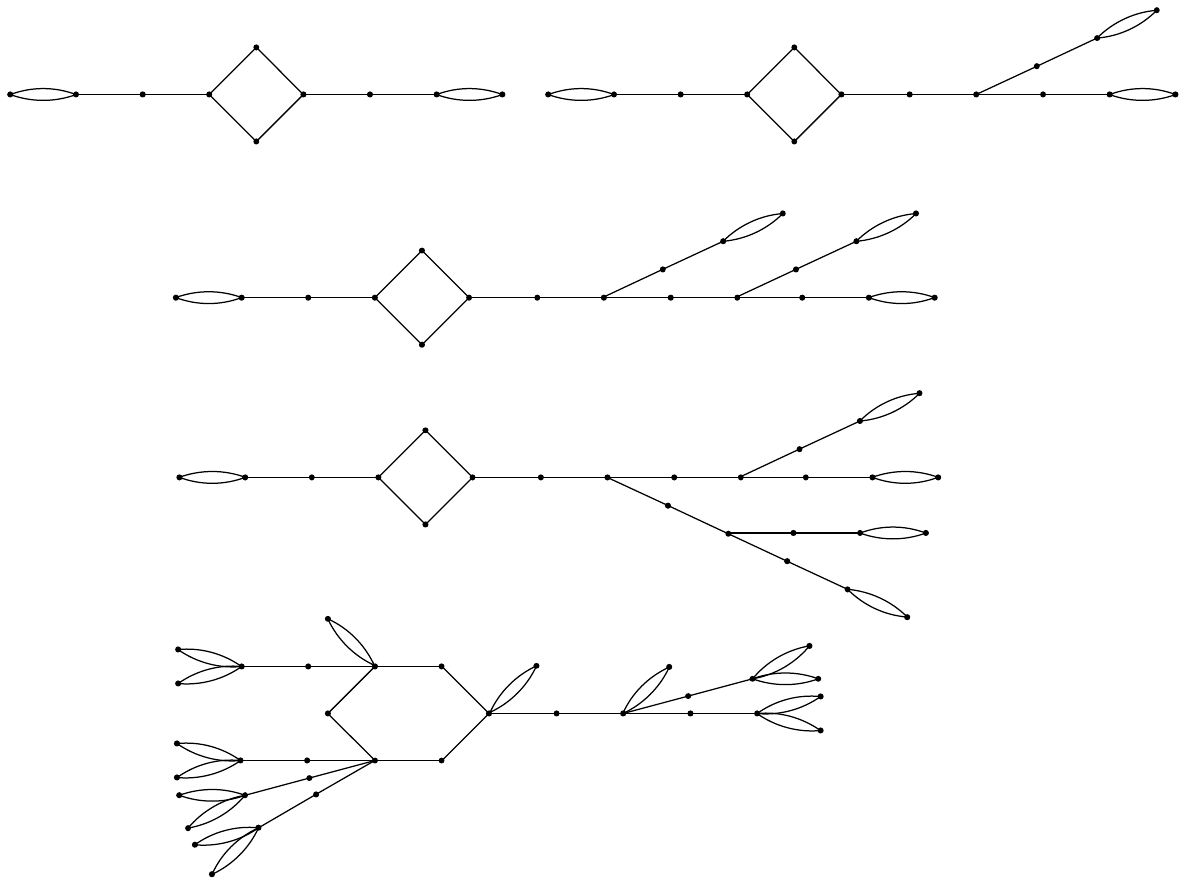}};
		\begin{scope}[
			x={(image.south east)},
			y={(image.north west)}
			]
			\node [black, font=\bfseries] at (0.23,0.8) {$H_1$};
            \node [black, font=\bfseries] at (0.69,0.8) {$H_2$};
            \node [black, font=\bfseries] at (0.37,0.57) {$H_3$};
			\node [black, font=\bfseries] at (0.37,0.37) {$H_4$};
            \node [black, font=\bfseries] at (0.35,0.07) {$H_5$};
		\end{scope}
	\end{tikzpicture}
	\caption{The graphs $H_1,H_2,H_3,H_4$ and $H_5$.}\label{H1-H5}
\end{figure}

It is easy to check that $0\cdot d(v)+b=s(v)$ for any vertex $v\in H$, where $H\in\mathcal{H}^1(b,t)$. Then we have the following result by Theorem \ref{lemma-s(v)}.
\begin{proposition}
  Let $H\in\mathcal{H}^1(b,t)$. Then $H$ has exactly two main eigenvalues.
\end{proposition}

\begin{theorem}\label{theorem-a=0}
  Let $H,a,b$ be as in proposition \ref{prop-a-b-fenlei-unicyclic}, $H^\circ\cong H'_{uv}T$, $v=v_l$, $u=v_{l+1}$, $Q=v_0v_1\cdots v_{l-1}v_l$ $(l\geq0)$ be the longest $v$-path in $T$. If $a=0$, then $H\in \mathcal{H}^1(b,t)$.
\end{theorem}

\begin{proof}
  According to the structure of $H^\circ$, we can assume that $V(H^\circ)=V(C)\cup V(T_1)\cup V(T_2)\cup\cdots \cup V(T_s)$ $(s\geq1)$. Without loss of generality, we can suppose $N(u)=\{u_1,u_2,p_1,p_2,\ldots,p_t\}$ $(t\leq s)$ and $H^\circ\cong {H^i}_{up_i}T_i$ for $1\leq i\leq t$, where $u,u_1,u_2\in V(C)$ and $p_i\in V(T_i)$.

  Let $i\in\{1,2,\ldots,t\}$. If $Q=v_0v_1\cdots v_{l-1}v_l$ $(l\geq0)$ is the longest $p_i$-path in $T_i$, where $p_i=v_l$, $u=v_{l+1}$, then by Lemma \ref{lemma-a=0}, we know that $l$ is even, $d(u)=\frac{b}{2}$, $d(p_i)=2$, and
  $
  w(p_i,u)=\left\{
    \begin{array}{ll}
      2, &\mbox{if $l=0$,}\\
      1, &\mbox{if $l\geq2$.}
    \end{array}\right.
  $
  Let $c_1,c_2$ be the number of vertices in $N(u)\backslash\{u_1,u_2\}$ with $w(u,p_i)=1$, $w(u,p_i)=2$, respectively. Then by \eqref{eq-s(v)-2}, we have $b=s(u)=w(u,u_1)d(u_1)+w(u,u_2)d(u_2)+2c_1+4c_2$ and $\frac{b}{2}=d(u)=w(u,u_1)+w(u,u_2)+c_1+2c_2$, which implies $d(u_1)=d(u_2)=2$. Thus, $w(u,u_1)=w(u,u_2)=1$ by $d(u_1)\geq2$ and $d(u_2)\geq2$.

  Let $N(u_2)=\{u,u_3\}$. Then $w(u,u_2)=w(u_2,u_3)=1$ by $d(u_2)=2$ and $w(u,u_2)=1$. By \eqref{eq-s(v)-2}, we have $b=s(u_2)=w(u,u_2)d(u)+w(u_2,u_3)d(u_3)$, which implies $d(u_3)=\frac{b}{2}$. Let $N(u_3)=\{u_2,u_4,q_1,q_2,\ldots,q_{t'}\}$ $(t'\geq1)$, where $u_4\in V(C)$. Then by a discussion similar to $u$, we have $d(u_4)=2$ and $w(u_3,u_4)=w(u_4,u_5)=1$. And so on, it is easy to verify that the degrees of vertices of $C$ are either $2$ or $\frac{b}{2}$, and they appear alternately. Not only that, but the vertices with a degree of $\frac{b}{2}$ in $C$ share the same structural characteristics as vertex $u$.

  On the other hand, the structure of the pendent trees of $H$ can be derived from Lemma \ref{lemma-a=0}. Thus, $H\in \mathcal{H}^1(b)$.
\end{proof}

\subsubsection{$a=1$ and $b\neq0$}

\begin{lemma}\label{lemma-a=1-1}
  Let $H,a,b$ be as in Proposition \ref{prop-a-b-fenlei-unicyclic}, $H^\circ\cong H'_{uv}T$. If $a=1$ and $b\neq0$, then $l\leq1$, moreover, if $l=1$, then $v$ and $u$ satisfies condition $(2)$ of Lemma $\ref{prop-type}$.
\end{lemma}

\begin{proof}
  We complete the proof by the following four cases.

  \textbf{Case 1.} $v_1,v_2$ is of type (1).

  Suppose to the contrary that $l\geq2$. By the proof of the Case 1 of Theorem 4.12 in \cite{Du-4}, we know that $d(v_1)=d(v_2)=1+b$ and any vertex $v'\in N(v_2)\backslash\{v_1\}$ is a pendent vertex with $w(v',v_2)=1$. However, by the structure of $H$, it is clear that there exist at least one vertex $v'\in N(v_2)$ with $d(v')\geq2$, it is a contradiction.

  Furthermore, if $l=1$, then $d(u)=d(v_1)=1+b$ and any vertex $v'\in N(u)\backslash\{v_1\}$ is a pendent vertex with $w(v',v_2)=1$. However, this contradicts the fact that there exists a $v'$ with $d(v')\geq2$ that belongs to the cycle.

  \textbf{Case 2.} $v_1,v_2$ is of type (3).

  Suppose to the contrary that $l\geq2$. By the proof of the Case 2 of Theorem 4.12 in \cite{Du-4}, we know that $d(v_1)=1+b$, $d(v_2)=1+\frac{b}{2}$ and $1+\frac{3}{2}b=s(v_2)<2d(v_1)=2+2b$, a contradiction. If $l=1$, then we have $d(u)=1+\frac{b}{2}$ and $s(u)<2d(v_1)=2+2b$ by a similar reason, which leads to a contradiction.

  \textbf{Case 3.} $v_1,v_2$ is of type (2) or (4).

  \textbf{Subcase 3.1.} $v_1,v_2$ is of type (2) and $l\geq2$.

  By the proof of the Case 3 of Theorem 4.12 in \cite{Du-4}, we know that $l\leq3$ and $w(v_2,v_3)=d(v_3)=2$. Now $v_3=u$, which contradicts $d(u)\geq3$.

  \textbf{Subcase 3.2.} $v_1,v_2$ is of type (4) and $l\geq2$.

  By the proof of the Case 4 of Theorem 4.12 in \cite{Du-4}, we know that $d(v_3)=1$. Obviously, $H$ cannot exist.

  \textbf{Subcase 3.3.} $v_1,v_2$ is of type (4) and $l=1$.

  Now $v_1=v$, $v_2=u$. By the discussion of Case 1, Case 2, Subcase 3.1 and Subcase 3.2 above, we know that there is no path $v'_0v'_1\cdots v'_mv'_{m+1}$ with $m\geq2$, where $v'_{m+1}=u$ and $v'_0$ is a pendent vertex. Moreover, if $m=1$, then $v'_1$ and $u$ is of type (4) by Lemma \ref{lem-different-1-4}.

  Let $V(C)\cap N(u)=\{u_1,u_2\}$, where $u=v_2$. By the proof of the Case 4 of Theorem 4.12 in \cite{Du-4}, we know that $d(v_1)=1+\frac{b}{2}$, $d(v_2)=\frac{3}{2}+\frac{b}{4}$ and any vertex $v'\in N(v_2)\backslash\{u_1,u_2\}$ cannot be a pendent vertex. Thus, combining the above discussion, we know that any vertex $v'\in N(v_2)\backslash\{u_1,u_2\}$ is a $(v_2,2)$-vertex with $d(v')=d(v_1)$ and $w(v',v_2)=2$. Then by \eqref{eq-s(v)-2}, we have $\frac{3}{2}+\frac{5b}{4}=d(u)+b=s(u)=w(u_1,u)d(u_1)+w(u_2,u)d(u_2)
  +2(1+\frac{b}{2})(\frac{d(u)-w(u_1,u)-w(u_2,u)}{2})$. By $\frac{3}{2}+\frac{5b}{4}=s(u)>2(1+\frac{b}{2})(\frac{d(u)-w(u_1,u)-w(u_2,u)}{2})$, it is clear that $\frac{d(u)-w(u_1,u)-w(u_2,u)}{2}=1$. Thus, we have $\frac{3}{2}+\frac{b}{4}=d(u)=w(u_1,u)+w(u_2,u)+2$ and $\frac{3}{2}+\frac{5b}{4}=s(u)=w(u_1,u)d(u_1)+w(u_2,u)d(u_2)
  +2(1+\frac{b}{2})$, which implies $w(u_1,u)(d(u_1)-1)+w(u_2,u)(d(u_2)-1)=0$. However, this implies $d(u_1)=d(u_2)=1$, which contradicts $d(u_1)\geq2$ and $d(u_2)\geq2$.

  \textbf{Subcase 3.4.} $v_1,v_2$ is of type (2) and $l=1$.

  Now $v_1=v$, $v_2=u$. By the discussion of Cases 1-2 and Subcases 3.1-3.3 above, we know that there is no path $v'_0v'_1\cdots v'_mv'_{m+1}$ with $m\geq2$, where $v'_{m+1}=u$ and $v'_0$ is a pendent vertex. Moreover, if $m=1$, then $v'_1$ and $v'_2$ is of type (2) by Lemma \ref{lem-different-1-4}.

  By the proof of the Case 3 of Theorem 4.12 in \cite{Du-4}, we know that $d(v_1)=d(v_2)=1+\frac{b}{2}$ and the vertex $v'\in N(v_2)\backslash\{v_1,u_1,u_2\}$ can be a pendent vertex with $w(v',v_2)=d(v')=2$, where $v_2=u$ and $u_1,u_2\in N(u)\cap V(C)$.
\end{proof}
\vskip0.2cm

Let $D_1$ and $D_2$ be $(0,1,2)$-multigraphs shown as in Figure \ref{H2-H3}. For $D_1$, we have $d(v_1)=d(v_3)=1$, $b=4k$ $(k\geq1)$, $d(v_2)=d(v_4)=1+\frac{b}{2}$ is odd, $w(v_1,v_2)=w(v_2,v_3)=w(v_2,v_4)=1$, any vertex $v\in N(v_4)\backslash\{v_2\}$ is a pendent vertex with $d(v)=w(v,v_4)=2$, and any vertex $v'\in N(v_2)\backslash\{v_1,v_3,v_4\}$ is a pendent vertex with $d(v)=w(v,v_4)=2$. For $D_2$, we have $d(v_1)=d(v_4)=1$, $b=4k+2$ $(k\geq1)$, $d(v_2)=d(v_3)=1+\frac{b}{2}$ is even, $w(v_1,v_2)=w(v_2,v_3)=w(v_3,v_4)=1$, any vertex $v\in N(v_2)\backslash\{v_1,v_3\}$ is a pendent vertex with $d(v)=w(v,v_2)=2$, and any vertex $v\in N(v_3)\backslash\{v_2,v_4\}$ is a pendent vertex with $d(v)=w(v,v_3)=2$.

Now, we define two types of $(0,1,2)$-multigraphs, denoted as $\mathcal{H}^2(b,t)$ and $\mathcal{H}^3(b,t)$ ($t\geq3$). Among them, $\mathcal{H}^2(b,t)$ (similarly for $\mathcal{H}^3(b,t)$) is constructed by cyclically connecting $t$ number of $D_1$ ($D_2$) horizontally. Specifically, the right side of one $D_1$ ($D_2$) is connected to the left side of another $D_1$ ($D_2$), repeating until all $t$ are connected, preserving each $D_1$ ($D_2$) internal connections.

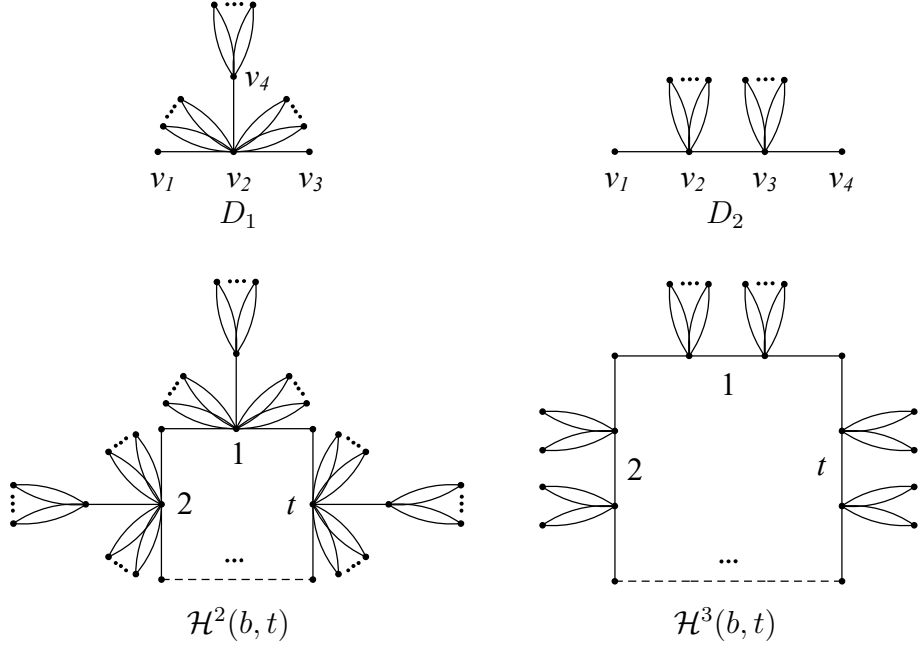
\begin{figure}[!h]
	\centering
	\begin{tikzpicture}
		\node[anchor=south west,inner sep=0] (image) at (0,0) {\includegraphics[width=0.8\textwidth]{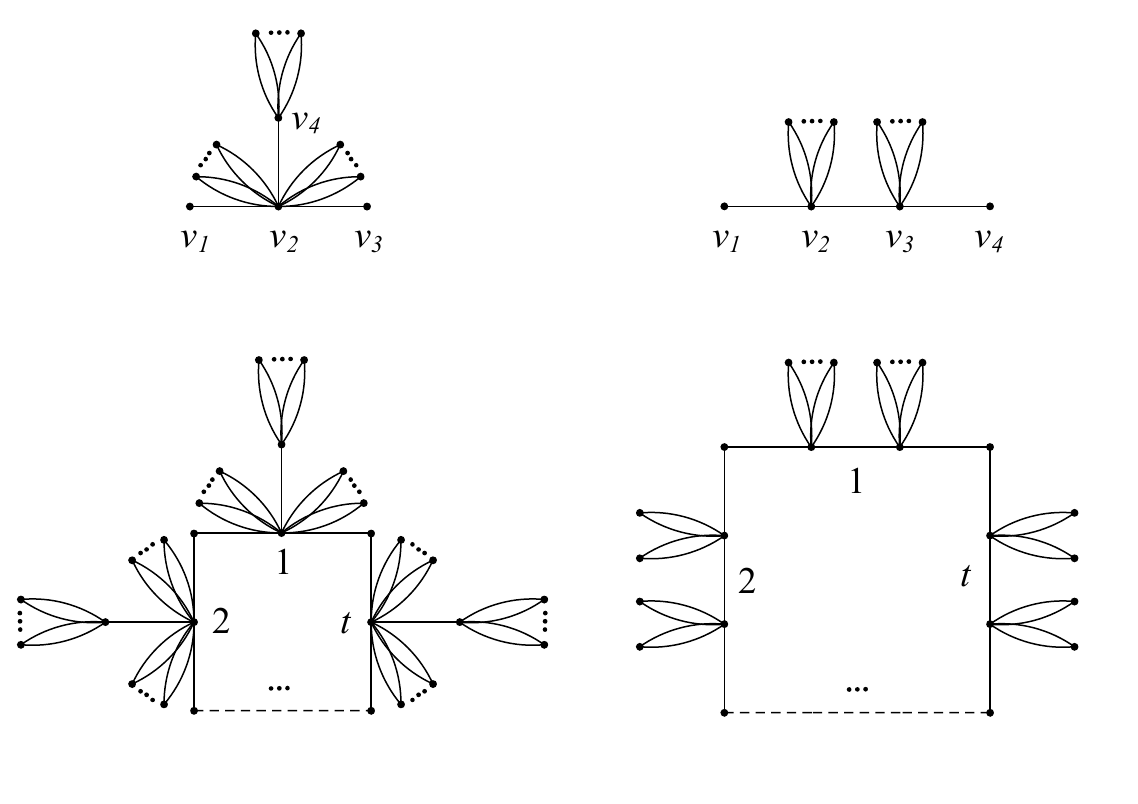}};
		\begin{scope}[
			x={(image.south east)},
			y={(image.north west)}
			]
			\node [black, font=\bfseries] at (0.25,0.65) {$D_1$};
            \node [black, font=\bfseries] at (0.75,0.65) {$D_2$};
            \node [black, font=\bfseries] at (0.25,0.05) {$\mathcal{H}^2(b,t)$};
			\node [black, font=\bfseries] at (0.75,0.05) {$\mathcal{H}^3(b,t)$};
		\end{scope}
	\end{tikzpicture}
	\caption{The graphs $D_1$, $D_2$, $\mathcal{H}^2(b,t)$ and $\mathcal{H}^3(b,t)$, where $b\geq2$ and $t\geq3$.}\label{H2-H3}
\end{figure}

It is easy to check that $1\cdot d(v)+b=s(v)$ for any vertex $v\in H$, where $H\in\{\mathcal{H}^2(b,t)\}$ (or $H\in\mathcal{H}^3(b,t)\}$). Then we have the following result by Theorem \ref{lemma-s(v)}.
\begin{proposition}
  Let $H\in\{\mathcal{H}^2(b,t),\mathcal{H}^3(b,t)\}$. Then $H$ has exactly two main eigenvalues.
\end{proposition}

\begin{theorem}\label{theorem-a=1-1}
  Let $H,a,b$ be as in Proposition \ref{prop-a-b-fenlei-unicyclic}. If $a=1$, then $H\in\{\mathcal{H}^2(b,t),\mathcal{H}^3(b,t)\}$ $(t\geq3)$.
\end{theorem}

\begin{proof}
We complete the proof by the following two cases.

  \textbf{Case 1.} there exist a pendent tree $T'$ on $C$ with $V(T')>1$.

  Let $H^\circ\cong H'_{uv}T'$, $Q=v_0v_1\cdots v_l$ $(v_l=v)$ be the longest $v$-path of $T'$. By Lemma \ref{lemma-a=1-1}, we have $l=1$, $d(u)=d(v)=1+\frac{b}{2}$, $w(u,v)=1$ and $v,u$ satisfies condition (2) of Proposition \ref{prop-type}, where $1+\frac{b}{2}$ $(\geq3)$ is odd. Moreover, if there exists a pendent vertex $v'\in N(u)$, then $d(v')=w(v',u)=2$. Let $u_1,u_2\in N(u)\cap V(C)$, $c_1$ and $c_2$ ($\geq1$) be the number of pendent vertices, non-pendent vertices of $N(u)\backslash \{u_1,u_2\}$, respectively.

  By \eqref{eq-s(v)-2}, we have $1+\frac{3}{2}b=d(u)+b=s(u)=4c_1+(1+\frac{b}{2})c_2+w(u,u_1)d(u_1)
  +w(u,u_2)d(u_2)$. It is clear that $c_2\leq2$ by $1+\frac{3}{2}b>(1+\frac{b}{2})c_2$.

  \textbf{Subcase 1.1.} $c_2=2$.

  Now $1+\frac{b}{2}=d(u)=2c_1+2+w(u,u_1)+w(u,u_2)$. By \eqref{eq-s(v)-2}, we have  $1+\frac{3}{2}b=s(u)=4c_1+2(1+\frac{b}{2})+w(u,u_1)d(u_1)
  +w(u,u_2)d(u_2)$, which implies $\frac{b}{2}-1=4c_1+w(u,u_1)d(u_1)
  +w(u,u_2)d(u_2)$. Since $\frac{b}{2}-1$ is odd, $w(u,u_1)=w(u,u_2)=2$ is impossible.

  If $w(u,u_1)=w(u,u_2)=1$, then $d(u_1)\geq2$ and $d(u_2)\geq2$. Now $1+\frac{b}{2}=d(u)=2c_1+2+1+1$, which implies $2c_1=\frac{b}{2}-3$. However, $4c_1+w(u,u_1)d(u_1)
  +w(u,u_2)d(u_2)\geq (b-6)+2+2=b-2>\frac{b}{2}-1$, a contradiction.

  If $w(u,u_1)=1$ and $w(u,u_2)=2$, then $d(u_1)\geq2$ and $d(u_2)\geq3$. Now $1+\frac{b}{2}=d(u)=2c_1+2+1+2$, which implies $2c_1=\frac{b}{2}-4$. However, $4c_1+w(u,u_1)d(u_1)
  +w(u,u_2)d(u_2)\geq (b-8)+2+2\times3=b>\frac{b}{2}-1$, a contradiction.

  If $w(u,u_1)=2$ and $w(u,u_2)=1$, then there is a contradiction through similar discussions.

  Thus, there is no such $H$.

  \textbf{Subcase 1.2.} $c_2=1$.

  By \eqref{eq-s(v)-2}, we have  $1+\frac{3}{2}b=s(u)=4c_1+(1+\frac{b}{2})+w(u,u_1)d(u_1)
  +w(u,u_2)d(u_2)$, which implies $b=4c_1+w(u,u_1)d(u_1)+w(u,u_2)d(u_2)$. Combining with $1+\frac{b}{2}=d(u)=2c_1+1+w(u,u_1)+w(u,u_2)$, we have $w(u,u_1)(d(u_1)-2)+w(u,u_2)(d(u_2)-2)=0$, which implies $d(u_1)=d(u_2)=2$. Then it is clear that $w(u,u_1)=w(u,u_2)=1$. Considering the similarity of $u_1$ and $u_2$, we will only study $u_2$ below.

  Since $d(u_2)=2$ and $u_2\in V(C)$, we can suppose $N(u_2)\backslash\{u\}=\{u_3\}$. It can be verified that $w(u_2,u_3)=1$. Then by \eqref{eq-s(v)-2}, we have $2+b=s(u_2)=w(u,u_2)d(u)+w(u_2,u_3)d(u_3)$, which implies $d(u_3)=1+\frac{b}{2}$.

  Suppose that $N(u_3)\cap V(C)\backslash\{u_2\}=\{u_4\}$. By Lemma \ref{lemma-a=1-1}, we know that any vertex $v'\in N(u_3)\backslash\{u_2,u_4\}$ is a pendent vertex or a $(u_3,2)$-vertex with $d(v')=1+\frac{b}{2}$ and $w(v',u_3)=1$. In fact, if $v'$ is a pendent vertex, then $d(v')=w(v',u_3)=2$ by (1) of Lemma \ref{lemma-dv1-dv2} and $b\neq0$.

  Let $c_1',c_2'$ be the number of pendent vertices, non-pendent vertices of $N(u_3)\backslash\{u_2,u_4\}$, respectively. Then by \eqref{eq-s(v)-2}, we have $1+\frac{3}{2}b=s(u_3)=4c_1'+c_2'(1+\frac{b}{2})+w(u_2,u_3)d(u_2)+
  w(u_3,u_4)d(u_4)$. Combining with $1+\frac{b}{2}=d(u_3)=2c_1'+c_2'+w(u_2,u_3)+w(u_3,u_4)$, we have $\frac{b}{2}-1=c_2'(\frac{b}{2}-1)+w(u_3,u_4)(d(u_4)-2)$. By $\frac{b}{2}-1\geq1$, it is clear that $c_2'\leq1$.

  If $c_2'=1$, then $w(u_3,u_4)(d(u_4)-2)=0$, which implies $d(u_4)=2$, and thus $w(u_3,u_4)=1$. It is not difficult to find that $u_4$ are similar to $u_2$, then we can study $u_4$ by the similar way that we study $u_2$.

  If $c_2'=0$, then $\frac{b}{2}-1=w(u_3,u_4)(d(u_4)-2)$. Since $\frac{b}{2}-1$ is odd, we know that $w(u_3,u_4)=1$. However, by $w(u_2,u_3)=w(u_3,u_4)=1$ and any vertex $v'\in N(u_3)\backslash\{u_2,u_4\}$ is a pendent vertex with $w(v,u_3)=2$, we have $d(u_3)$ is even, which contradicts with $d(u_3)=1+\frac{b}{2}$ is odd.

  Combining with the above discussion, we know that $H$ can be composed of repeated connections of $D_1$, and thus $H\in \mathcal{H}^2(b,t)$, where $b=4k$ $(\geq4)$ and $t\geq3$.

  \textbf{Case 2.} any pendent tree $T$ on $C$ satisfies $|V(T')|=1$.

  Let $H^\circ\cong H'_{uv}T'$. Then $V(T')=\{v\}$. Suppose that $N(u)\cap V(C)=\{u_1,u_2\}$, $N(u_1)\cap V(C)=\{u,u_0\}$ and $N(u_2)\cap V(C)=\{u_2,u_3\}$. Since any pendent tree $T$ on $C$ satisfies $|V(T)|=1$, we have $d(u_2)\in\{1+b,1+\frac{b}{2}\}$ if $N(u_2)\backslash\{u,u_3\}\neq\emptyset$ and $d(u_2)\in\{2,3,4\}$ if $N(u_2)\backslash\{u,u_3\}=\emptyset$.

  \textbf{Subcase 2.1.} $w(u,v)=2$.

  By $w(u,v)=2$, $b\neq0$ and Lemma \ref{lem-b=0}, we know that any vertex $v'\in N(u)\backslash\{u_1,u_2\}$ is a pendent vertex with $d(v')=w(v',u)=2$.
  By (1) of Lemma \ref{lemma-dv1-dv2}, we have $d(u)=1+\frac{b}{2}$ $(\geq4)$. By \eqref{eq-s(v)-2}, we have $1\cdot d(u)+b=s(u)=w(u,u_1)d(u_1)+w(u,u_2)d(u_2)+2\big(d(u)-w(u,u_1)-w(u,u_2)\big)$, which implies
  \begin{equation}\label{eq-a=1-1}
    \frac{b}{2}-1=w(u,u_1)\big(d(u_1)-2\big)+w(u,u_2)\big(d(u_2)-2\big).
  \end{equation}

  \textbf{Subcase 2.1.1.} $N(u_2)\backslash\{u,u_3\}\neq\emptyset$.

  If any vertex $p\in N(u_2)\backslash\{u,u_3\}$ is a pendent vertex with $d(p)=w(p,u_2)=1$, then $d(u_2)=1+b$ by (1) of Lemma \ref{lemma-dv1-dv2}. However, $w(u,u_1)\big(d(u_1)-2\big)+w(u,u_2)\big(d(u_2)-2\big)>
  w(u,u_2)(b-1)\geq b-1>\frac{b}{2}-1$ by $1+\frac{b}{2}\geq4$, which contradicts with \eqref{eq-a=1-1}.

  If any vertex $p\in N(u_2)\backslash\{u,u_3\}$ is a pendent vertex with $d(p)=w(p,u_2)=2$, then $d(u_2)=1+\frac{b}{2}$ by (1) of Lemma \ref{lemma-dv1-dv2}. By \eqref{eq-a=1-1}, we have $\frac{b}{2}-1=w(u,u_1)\big(d(u_1)-2\big)+w(u,u_2)(\frac{b}{2}-1)$, and this equation holds only if $d(u_1)=2$ and $w(u,u_2)=1$. Then it is clear that $w(u,u_1)=1$.

  Since $d(u_2)=1+\frac{b}{2}\geq4$, $w(u,u_2)=1$ and $w(u_2,u_3)\leq2$, we have $N(u_2)\backslash\{u,u_3\}\neq\emptyset$, moreover, any vertex $p\in N(u_2)\backslash\{u,u_3\}$ is a pendent vertex with $d(p)=w(p,u_2)=2$ by $d(u_2)=1+\frac{b}{2}$, the structure of $H$ and (1) of Lemma \ref{lemma-dv1-dv2}. Then by \eqref{eq-s(v)-2}, we have $1\cdot d(u_2)+b=s(u_2)=w(u_2,u)d(u)+w(u_2,u_3)d(u_3)+2\big(d(u_2)-w(u_2,u)
  -w(u_2,u_3)\big)$, which implies $w(u_2,u_3)(d(u_3)-2)=0$. Thus, $d(u_3)=2$, and then $w(u_2,u_3)=w(u_3,u_4)=1$, where $\{u_4\}=N(u_3)\backslash\{u_2\}$.

  By \eqref{eq-s(v)-2}, we have $1\cdot d(u_3)+b=s(u_3)=w(u_2,u_3)d(u_2)+w(u_3,u_4)d(u_4)$, which implies $d(u_4)=1+\frac{b}{2}$. It can be verified that $u_4$ is similar to $u$, and $u_1$ is similar to $u_3$, then we can study $u_4$, $u_1$ in a similar way. By analogy, it is not difficult to find that $H$ can be composed of $D_2$. Thus, $H\in \mathcal{H}^3(b,t)$, where $b=4k+2$ $(\geq6)$ and $t\geq3$.

  \textbf{Subcase 2.1.2.} $N(u_1)\backslash\{u,u_0\}\neq\emptyset$.

  By a discussion similar to Subcase 2.1.1, we have $H\in \mathcal{H}^3(b,t)$.

  \textbf{Subcase 2.1.3.} $N(u_1)\backslash\{u,u_0\}=\emptyset$ and $N(u_2)\backslash\{u,u_3\}=\emptyset$.

  If $w(u,u_1)=2$, then $d(u_1)\in\{3,4\}$ by $N(u_1)\backslash\{u,u_0\}=\emptyset$. However, by $1\cdot d(u_1)+b=s(u_1)=w(u_1,u)d(u)+w(u_0,u_1)d(u_0)$, it is easy to check that $d(u_0)=1$ for $d(u_1)\in\{3,4\}$, and this is a contradiction.

  If $w(u,u_2)=2$, then there is a contradiction through a similar discussion.

  If $w(u,u_1)=w(u,u_2)=1$, then $d(u)=1+\frac{b}{2}$ $(\geq4)$ is even. By \eqref{eq-a=1-1}, we have $\frac{b}{2}-1=d(u_1)-2+d(u_2)-2$. Since $\frac{b}{2}-1$ is even and $N(u_1)\backslash\{u,u_0\}=N(u_2)\backslash\{u,u_3\}=\emptyset$, we have $d(u_1)=d(u_2)=2$ or $d(u_1)=d(u_2)=3$. If $d(u_1)=d(u_2)=2$, then $\frac{b}{2}-1=0$, which contradicts $1+\frac{b}{2}\geq4$. If $d(u_1)=d(u_2)=3$, then $\frac{b}{2}-1=2$ and $w(u_2,u_3)=2$. However, by \eqref{eq-s(v)-2} we have $1\cdot d(u_2)+b=s(u_2)=w(u,u_2)d(u_2)+w(u_2,u_3)d(u_3)$, and this implies $d(u_3)=1+\frac{b}{4}$ is not an integer, a contradiction.

  \textbf{Subcase 2.2.} $w(u,v)=1$.

  By $w(u,v)=1$, $b\neq0$ and Lemma \ref{lem-b=0}, we know that any vertex $v'\in N(u)\backslash\{u_1,u_2\}$ is a pendent vertex with $d(v')=w(v',u)=1$.
  By (1) of Lemma \ref{lemma-dv1-dv2}, we have $d(u)=1+b$ $(\geq3)$. By \eqref{eq-s(v)-2}, we have $1\cdot d(u)+b=s(u)=w(u,u_1)d(u_1)+w(u,u_2)d(u_2)+\big(d(u)-w(u,u_1)-w(u,u_2)\big)$, which implies
  \begin{equation}\label{eq-a=1-2}
    b=w(u,u_1)\big(d(u_1)-1\big)+w(u,u_2)\big(d(u_2)-1\big).
  \end{equation}

  \textbf{Subcase 2.2.1.} $N(u_2)\backslash\{u,u_3\}\neq\emptyset$.

  If any vertex $p\in N(u_2)\backslash\{u,u_3\}$ is a pendent vertex with $d(p)=w(p,u_2)=1$, then $d(u_2)=1+b$ by (1) of Lemma \ref{lemma-dv1-dv2}. However, it is clear that $w(u,u_1)\big(d(u_1)-1\big)+w(u,u_2)\big(d(u_2)-1\big)>
  w(u,u_2)b\geq b$, which contradicts with \eqref{eq-a=1-2}.

  If any vertex $p\in N(u_2)\backslash\{u,u_3\}$ is a pendent vertex with $d(p)=w(p,u_2)=2$, then $d(u_2)=1+\frac{b}{2}$ by (1) of Lemma \ref{lemma-dv1-dv2}. However, it can be verified that $1\cdot d(u_2)+b=s(u_2)<w(u,u_2)(d(u)-2)+2d(u_2)<w(u,u_2)d(u)+w(u_2,u_3)d(u_3)+2\big(d(u_2)-w(u,u_2)
  -w(u_2,u_3)\big)$, which contradicts the result obtained by applying \eqref{eq-s(v)-2} to $u_2$.

  \textbf{Subcase 2.2.2.} $N(u_1)\backslash\{u,u_0\}\neq\emptyset$.

  By a discussion similar to Subcase 2.2.1, we know that there is no such $H$.

  \textbf{Subcase 2.2.3.} $N(u_1)\backslash\{u,u_0\}=\emptyset$ and $N(u_2)\backslash\{u,u_3\}=\emptyset$.

  If $w(u_1,u)=2$, then $d(u_1)=w(u_1,u)+w(u_0,u_1)=2+w(u_0,u_1)\leq4$ by $N(u_1)=\{u,u_0\}$. It is easy to check that $1\cdot d(u_1)+b=s(u_1)<w(u_1,u)d(u)+w(u_0,u_1)d(u_0)$, and this implies a contradiction when applying \eqref{eq-s(v)-2} to $u_1$.

  If $w(u,u_2)=2$, then there is a contradiction through a similar discussion.

  If $w(u,u_1)=w(u,u_2)=1$, then $d(u_2)=w(u,u_2)+w(u_2,u_3)=1+w(u_2,u_3)$ by $N(u_2)=\{u,u_3\}$. It can be verified that $1+w(u_2,u_3)+b=1\cdot d(u_2)+b=s(u_2)\leq w(u_2, u)d(u_2)+w(u_2,u_3)d(u_3)=1+b+w(u_2,u_3)d(u_3)$. Importantly, the equality is valid only when $d(u_3)=1$. But since we know $d(u_3)\geq2$ and if we applying \eqref{eq-s(v)-2} to $u_2$, we'll find that this leads to a contradiction.
\end{proof}

\section{Conclusion}\label{sec-conclusion}

In this paper, we characterize signed graphs with exactly two main eigenvalues by studying $(0,1,2)$-multigraphs $H$ whose $B$-graphs are unicyclic. If the integers $a,b$ satisfies \eqref{eq-s(v)-2} for any $v\in V(H)$, then $H$ can be analyzed through the following four cases: (1) $a=0$ and $b\neq0$; (2) $a=1$ and $b\neq0$; (3) $a\geq2$ and $b\neq0$; (4) $a\neq0$ and $b=0$.

We have completely resolved Cases (1) and (2), and partially resolved Case (3). These results are summarized in Table \ref{tab-summarize}. Cases (3) (Problem \ref{problem-1}) and (4) (Problem \ref{problem-2}) remain open for further research.

\begin{table}[h]
\renewcommand{\arraystretch}{1.5}\caption{Known results.}\label{tab-summarize}\vskip0.05cm
\centering
{
\begin{tabular}{cl}\hline\hline
$(a,b)$& $H$ \\
\hline\hline
(1) $a=0$, $b>0$& $H\in \mathcal{H}^1(b,t)$, where $b=4k+2\geq 6$ and $t\geq4$ are even.\\
\hline
(2) $a=1$, $b\neq0$& $H\in \mathcal{H}^2(b,t)$ with $b=4k\geq4$ and $t\geq3$, or $H\in\mathcal{H}^3(b,t)$ with\\
                   & $b=4k+2\geq6$ and $t\geq3$, or $H\cong U^3_{3t}$ with $t\geq1$.\\
\hline
(3) $a\geq2$, $b\neq0$ & $H\in\{U^1_{4t},U^2_{3t},U^4_{5t},U^5_{4t}\}$ $(t\geq1)$ if the $B$-graph of $H$ is a cycle,\\
                       & and the problem is unsolved otherwise.\\
\hline
(4) $a>0$, $b=0$& No such $H$ exists if the $B$-graph of $H$ is a cycle, and the problem \\
                & is unsolved otherwise.\\
\hline\hline
\end{tabular}}
\end{table}

\begin{problem}\label{problem-1}
  Let $H,a,b$ be as in Proposition $\ref{prop-a-b-fenlei-unicyclic}$. Characterize $H$ that satisfies $a\geq2$ and $b\neq0$.
\end{problem}

\begin{problem}\label{problem-2}
  Let $H,a,b$ be as in Proposition $\ref{prop-a-b-fenlei-unicyclic}$. Characterize $H$ that satisfies $a>0$ and $b=0$.
\end{problem}

%

\noindent
{\bf Declaration of competing interest}\,

The authors declare that they have no known competing financial interests or personal relationships that could have appeared to influence the work reported in this paper.

\noindent
{\bf Data availability}\,

No data was used for the research described in the article.

\noindent
{\bf Acknowledgements}\,

This work is supported by the National Natural Science Foundation of China (Grant No. 12501491), the Fundamental Research Program of Shanxi Province (Grant No. 202503021212100), the Scientific and Technological Innovation Programs of Higher Education Institutions in Shanxi, the Natural Science Foundation of Hubei Province (No. 2025AFD006), the Foundation of Hubei Provincial Department of Education (No. Q20232505), and the Open Project Program Key Laboratory of Operations Research and Cybernetics of Fujian Universities, Fuzhou University (Grant No. G20240905).



\end{spacing}

\begin{thebibliography}{99}
\bibitem{Akbari} S. Akbari, F.A.M. Fran\c{c}a, E. Ghasemian, M. Javarsineh, L.S. de Lima, The main eigenvalues of signed graphs, Linear Algebra Appl., 614 (2021) 270--280.

\bibitem{Andelic} M. An$\mathrm \dbar$eli\'{c}, T. Koledin and Z. Stani\'{c}, Signed graphs whose all Laplacian eigenvalues are main, Linear and Multilinear Algebra, (2023) 2409--2425.

\bibitem{Cvetkovic-Main} D. Cvetkovi\'c, The main part of spectrum, divisors and switching of graphs, Publ. Inst. Math. (Beograd), 23 (1978) 31--38.

\bibitem{Cve-Controllable}D. Cvetkovi\'c, P. Rowlinson, Z. Stani\'c, M.G. Yoon, Controllable graphs, Bull. Acad. Serbe Sci. Arts. Classe Sci. Math. Natur. Sci. Math., 140 (2011) 81--88.

\bibitem{Cve-Controllable-2}D. Cvetkovi\'c, P. Rowlinson, Z. Stani\'c, M.G. Yoon, Controllable graphs with least eigenvalue at least $-2$, Appl. Anal. Discrete Math., 5 (2011) 165--175.

\bibitem{Cve-group}D. Cvetkovi\'c, P.W. Fowler, A group-theoretical bound for the number of main eigenvalues of a graph, J. Chem. Inf. Comput. Sci., 39 (1999) 638--641.

\bibitem{Du-1}Z. Du, F. Liu, S. Liu, Z. Qin, Graphs with $n-1$ main eigenvalues, Discrete Math., 344 (2021) 112397.

\bibitem{Du-2}Z. Du, L. You, H. Liu, F. Liu, Further results on almost controllable graphs, Linear Algebra Appl., 677 (2023) 31--50.

\bibitem{Du-3}Z. Du, L. You, H. Liu, Almost controllable graphs and beyond, Discrete Math., 347 (2024) 113743.

\bibitem{Du-4}Z. Du, L. You, H. Liu, X. Yuan, Signed graphs with exactly two main eigenvalues, Linear Algebra Appl., 695 (2024) 1--27.

\bibitem{Du-5}Z. Du, L. You, H. Liu, Further results on signed graphs with exactly two distinct main eigenvalues, Discrete Math., 349 (2026) 114712.

\bibitem{Feng}L. Feng, L. Lu, D. Stevanovi\'{c}, A short remark on graphs with two main eigenvalues, Appl. Math. Comput., 369 (2020) 124858.

\bibitem{Farrugia}A. Farrugia, On strongly asymmetric and controllable primitive graphs, Discrete Appl. Math., 211 (2016) 58--67.

\bibitem{Hou-bicyclic}Y. Hou, Z. Tang, W.C. Shiu, Some results on graphs with exactly two main eigenvalues, Appl. Math. Lett., 25 (2012) 1274--1278.

\bibitem{Hou-unicyclic}Y. Hou, F. Tian, Unicyclic graphs with exactly two main eigenvalues, Appl. Math. Lett., 19 (2006) 1143--1147.

\bibitem{Hou-trees}Y. Hou, H. Zhou, Trees with exactly two main eigenvalues, J. Nat. Sci. Hunan Norm. Univ., 26 (2005) 1--3 (in Chinese).

\bibitem{Hagos} E.M. Hagos, Some results on graph spectra, Linear Algebra Appl., 356 (2002) 103--111.

\bibitem{Hayat}S. Hayat, J.H. Koolen, F. Liu, Z. Qiao, A note on graphs with exactly two main eigenvalues, Linear Algebra Appl., 511 (2016) 318--327.

\bibitem{Li}S. Li, J. Wang, On the generalized $A_\alpha$-spectral characterizations of almost $\alpha$-controllable graphs, Discrete Math., 345 (2022) 112913.

\bibitem{Lepovic}M. Lepovi\'{c}, On eigenvalues and main eigenvalues of a graph, Math. Moravica, 4 (2000) 51--58.

\bibitem{Qiu}L. Qiu, W. Wang, W. Wang, H. Zhang, A new criterion for almost controllable graphs being determined by their generalized spectra, Discrete Math., 345 (2022) 113060.

\bibitem{Rowlinson-survey} P. Rowlinson, The main eigenvalues of a graph: a survey, Appl. Anal. Discrete Math. 1 (2007) 455--471.

\bibitem{Stanic} Z. Stani\'{c}, Main eigenvalues of real symmetric matrices with application to signed graphs, Czechoslovak Mathematical Journal, 70 (2020) 1091--1102.

\bibitem{Stanic-2}Z. Stani\'{c}, Further results on controllable graphs, Discrete Appl. Math., 166 (2014) 215--221.

\bibitem{Shao} Z. Shao, X. Yuan, Some signed graphs whose eigenvalues are main, Appl. Math. Comput., 423 (2022) 127014.

\bibitem{Wang}W. Wang, F. Liu, Generalized spectral characterizations of almost controllable graphs, European J. Combin., 96 (2021) 103348.
%
%
%
%

%
%
%
%
%

%

%
%
%
%
%
%
%
%

%
%

%
%

%



\end{thebibliography}
\end{document}